\documentclass[12pt]{amsart}
\usepackage{amsmath, amssymb, amsfonts, amsthm, fullpage, stmaryrd,  url}

\newtheorem{theorem}{Theorem}[subsection]
\newtheorem{lemma}[theorem]{Lemma}
\newtheorem{cor}[theorem]{Corollary}

\theoremstyle{definition}
\newtheorem{defn}[theorem]{Definition}
\newtheorem{example}[theorem]{Example}
\newtheorem{exercise}[theorem]{Exercise}
\newtheorem{hypothesis}[theorem]{Hypothesis}
\newtheorem{convention}[theorem]{Convention}
\newtheorem{remark}[theorem]{Remark}

\numberwithin{equation}{theorem}

\newcommand{\bA}{\mathbf{A}}
\newcommand{\be}{\mathbf{e}}
\newcommand{\Fp}{\mathbb{F}_p}
\newcommand{\Qp}{\mathbb{Q}_p}
\newcommand{\QQ}{\mathbb{Q}}
\newcommand{\Zp}{\mathbb{Z}_p}
\newcommand{\ZZ}{\mathbb{Z}}

\newcommand{\gothm}{\mathfrak{m}}
\newcommand{\gotho}{\mathfrak{o}}

\DeclareMathOperator{\FEt}{\mathbf{F\acute{E}t}}

\DeclareMathOperator{\Gal}{Gal}
\DeclareMathOperator{\GL}{GL}
\DeclareMathOperator{\Hom}{Hom}

\DeclareMathOperator{\Spec}{Spec}
\DeclareMathOperator{\Trace}{Trace}

\begin{document}

\title{New methods for $(\varphi, \Gamma)$-modules}
\author{Kiran S. Kedlaya}
\thanks{Supported by NSF CAREER grant DMS-0545904,
DARPA grant HR0011-09-1-0048, MIT (NEC Fund), UC San Diego
(Warschawski Professorship). Thanks to Ruochuan Liu, Ryan Rodriguez,
and Sarah Zerbes for helpful feedback.}
\date{January 14, 2015}

\begin{abstract}
We provide new proofs of two key results of $p$-adic Hodge theory: the Fontaine-Wintenberger isomorphism between Galois groups in characteristic 0 and characteristic $p$, and the Cherbonnier-Colmez theorem on decompletion of $(\varphi, \Gamma)$-modules. These proofs are derived from joint work with Liu on relative $p$-adic Hodge theory, and are closely related to the theory of perfectoid algebras and spaces, as in the work of Scholze.
\end{abstract}

\dedicatory{To Robert, forever quelling the rebellious provinces}

\maketitle

Let $p$ be a prime number. 
The subject of \emph{$p$-adic Hodge theory} concerns the interplay between different objects arising from
the cohomology of algebraic varieties over $p$-adic fields.
A good introduction to the subject circa 2010 
can be found in the notes of Brinon and Conrad
\cite{brinon-conrad}; however, in the subsequent years the subject has been radically altered by the introduction of some new ideas and techniques. While these ideas have their origins in work of this author on relative $p$-adic Hodge theory \cite{kedlaya-icm, kedlaya-witt} and were further developed in joint work with Liu \cite{kedlaya-liu1, kedlaya-liu2}, they are most widely known through Scholze's work on the theory and applications of \emph{perfectoid algebras} \cite{scholze1, scholze2, scholze-icm, scholze-torsion, scholze-weinstein}.

The purpose of this paper is to reinterpret and reprove two classic results of $p$-adic Hodge theory through the optic of perfectoid algebras (but in a self-contained manner).
The first of these results is a theorem of Fontaine and Wintenberger
\cite{fontaine-wintenberger} on the relationship between Galois theory in characteristic 0 and characteristic $p$.
\begin{theorem}[Fontaine-Wintenberger] \label{T:field of norms1}
For $\mu_{p^\infty}$ the group of all $p$-power roots of unity in an algebraic closure of $\Qp$,
the absolute Galois groups of the fields $\Fp((\overline{\pi}))$ and $\Qp(\mu_{p^\infty})$ are isomorphic (and even homeomorphic
as profinite topological groups).
\end{theorem}
The original proof of Theorem~\ref{T:field of norms1} depends in a crucial way on higher ramification theory of local
fields, as developed for instance in the book of Serre \cite{serre-local-fields}. This causes difficulties
when trying to generalize Theorem~\ref{T:field of norms1}, e.g., to local fields with imperfect residue fields.
We expose here a new approach to Theorem~\ref{T:field of norms1} in which ramification theory plays no role; one instead makes a careful analysis of rings of Witt vectors over valuation rings. In the process, we obtain a far more general result, in which $\Qp(\mu_{p^\infty})$ can be replaced by any sufficiently ramified $p$-adic field; 
more precisely, we obtain a functorial (and hence compatible with Galois theory) correspondence between \emph{perfectoid fields} and
perfect fields of characteristic $p$.
This is the \emph{tilting correspondence} in the sense of \cite{scholze1}, which is proved using almost ring theory; our proof here is the somewhat more elementary argument found in \cite{kedlaya-liu1} (see Remark~\ref{R:scholze} for further discussion). As a historical note, we remark that we learned the key ideas from this proof from attending Coleman's 1997 Berkeley course ``Fontaine's theory of the mysterious functor,'' whose principal content appears in \cite{coleman-iovita}.

Our second topic is the description of continuous representations of $p$-adic Galois groups on $\Qp$-vector spaces
(such as might arise from \'etale cohomology with $p$-adic coefficients)
in terms of \emph{$(\varphi, \Gamma)$-modules}. The original description of this form was given by Fontaine
\cite{fontaine-phigamma}
in terms of a Cohen ring for a field of formal power series, and is an easy consequence of
Theorem~\ref{T:field of norms1}. Our main focus is the refinement of Fontaine's result by Cherbonnier and Colmez
\cite{cherbonnier-colmez},
in which the Cohen ring is replaced with a somewhat smaller ring of convergent power series
(see Theorem~\ref{T:overconvergent2} for the precise statement).
This refinement is critical to a number of applications of $p$-adic Hodge theory,
notably Colmez's construction of the $p$-adic Langlands correspondence for $\GL_2(\QQ_p)$
\cite{colmez-langlands}).

Existing proofs of the Cherbonnier-Colmez theorem, including a generalization to families of representations
by Berger and Colmez \cite{berger-colmez}, rely on some calculations involving a formalism for
decompletion in continuous Galois cohomology, inspired by results of Tate and Sen and later axiomatized by Colmez.
However, one can express the proof in such a way that one makes essentially the same calculations on
$(\varphi, \Gamma)$-modules as in \cite{cherbonnier-colmez}, but without any need to 
introduce the Tate-Sen formalism. Besides making the proof more transparent,
this approach gives rise to analogous results for representations of the \'etale fundamental groups of some rigid analytic spaces;
for instance, the theory of overconvergent relative $(\varphi, \Gamma)$-modules introduced by
Andreatta-Brinon \cite{andreatta-brinon} is generalized in 
\cite{kedlaya-liu2} using this approach. (It should similarly be possible to recover the results of \cite{berger-colmez} in this fashion, and even to obtain a common generalization with \cite{kedlaya-liu2}.)

\section{Comparison of Galois groups}

\subsection{Preliminaries on strict $p$-rings}

We begin by recalling some basic properties of strict $p$-rings underlying the constructions made later,
following the derivations in \cite[\S 5]{serre-local-fields}.
(All rings considered will be commutative and unital.)

\begin{lemma} \label{L:theta map}
For any ring $R$ and any nonnegative integer $n$, the map $x \mapsto x^{p^n}$ induces a well-defined multiplicative monoid map
$\theta_n: R/(p) \to R/(p^{n+1})$.
\end{lemma}
\begin{proof}
If $x \equiv y \pmod{p^m}$ for some positive integer $m$, then $x^p - y^p = (x-y)(x^{p-1} + \cdots + y^{p-1})$
and the latter factor is congruent to $px^{p-1}$ modulo $p^m$; hence $x^p \equiv y^p \pmod{p^{m+1}}$. This
proves that $\theta_n$ is well-defined; it is clear that it is multiplicative.
\end{proof}

\begin{defn}
A ring $R$ of characteristic $p$ is \emph{perfect} if the Frobenius homomorphism $x \mapsto x^p$ is a bijection;
this  forces $R$ to be reduced. (If $R$ is a field, then $R$ is perfect if and only if every finite extension of $R$
is separable.)
A \emph{strict $p$-ring} is a $p$-torsion-free, $p$-adically complete ring $R$ for which $R/(p)$ is perfect,
regarded as a topological ring using the $p$-adic topology.
\end{defn}

\begin{example} \label{exa:p-rings}
The ring $\Zp$ is a strict $p$-ring with $\Zp/(p) \cong \Fp$. Similarly,
for any (possibly infinite) set $X$,
if we write $\ZZ[X]$ for the polynomial ring over $\ZZ$ generated by $X$
and put $\ZZ[X^{p^{-\infty}}] = \cup_{n=0}^\infty \ZZ[X^{p^{-n}}]$, then
the $p$-adic completion $R$ of $\ZZ[X^{p^{-\infty}}]$ 
is a strict $p$-ring with $R/(p) \cong \Fp[X^{p^{-\infty}}]$.
\end{example}

\begin{lemma} \label{L:Teichmuller1}
Let $\overline{R}$ be a perfect ring of characteristic $p$,
let $S$ be a $p$-adically complete ring, and let $\pi: S \to S/(p)$ be the natural projection.
Let $\overline{t}: \overline{R} \to S/(p)$ be a ring homomorphism. Then there exists a unique multiplicative map $t:
\overline{R} \to S$ with $\pi \circ t = \overline{t}$. In fact,
$t(\overline{x}) \equiv x_n^{p^n} \pmod{p^{n+1}}$ for any nonnegative integer $n$ and any
$x_n \in S$ lifting $\overline{t}(\overline{x}^{p^{-n}})$.
\end{lemma}
\begin{proof}
This is immediate from Lemma~\ref{L:theta map}.
\end{proof}

\begin{defn} \label{D:Teichmuller}
Let $R$ be a strict $p$-ring.
By the case $S=R$ of Lemma~\ref{L:Teichmuller1}, 
the projection $R \to R/(p)$ admits a unique multiplicative section $\left[\bullet \right]: R/(p) \to R$,
called the \emph{Teichm\"uller map}. (For example, the image of 
$\left[\bullet \right]: \Fp \to \Zp$ consists of 0 together with the $(p-1)$-st roots of unity in $\Zp$.)
Each $x \in R$ admits a unique representation as a $p$-adically convergent sum
$\sum_{n=0}^\infty p^n [\overline{x}_n]$ for some elements $\overline{x}_n \in R/(p)$, called the 
\emph{Teichm\"uller coordinates} of $x$.
\end{defn}

\begin{lemma} \label{L:Teichmuller2}
Let $R$ be a strict $p$-ring, let $S$ be a $p$-adically complete ring, and let $\pi: S \to S/(p)$ be the natural projection.
Let $t: R/(p) \to S$
be a multiplicative map such that $\overline{t} = \pi \circ t$ is a ring homomorphism. 
Then the formula
\begin{equation} \label{eq:Teichmuller homomorphism}
T\left( \sum_{n=0}^\infty p^n [\overline{x_n}] \right) = \sum_{n=0}^\infty p^n t(\overline{x}_n) \qquad (\overline{x}_0, \overline{x}_1, \dots \in R/(p))
\end{equation}
defines a (necessarily unique) $p$-adically continuous
homomorphism $T: R \to S$ such that
$T \circ \left[\bullet \right] = t$.
\end{lemma}
\begin{proof}
We check by induction that for each positive integer $n$, $T$ induces an additive map $R/(p^n) \to S/(p^n)$.
This holds for $n=1$ because $\pi \circ t$ is a homomorphism. Suppose the claim holds for some $n \geq 1$.
For $x = [\overline{x}] + p x_1, y = [\overline{y}] + py_1, z = [\overline{z}] + pz_1 \in R$
with $x+y=z$,
\begin{align*}
[\overline{z}] &\equiv ([\overline{x}^{p^{-n}}] + [\overline{y}^{p^{-n}}])^{p^{n}} \pmod{p^{n+1}} \\
t(\overline{z}) &\equiv (t(\overline{x}^{p^{-n}}) + t(\overline{y}^{p^{-n}}))^{p^{n}} \pmod{p^{n+1}}
\end{align*}
by Lemma~\ref{L:Teichmuller1}. In particular,
\begin{equation} \label{eq:teichmuller1}
T([\overline{z}]) - T([\overline{x}]) - T([\overline{y}]) \equiv
\sum_{i=1}^{p^{n}-1}
 \binom{p^{n}}{i} t(\overline{x}^{ip^{-n}} \overline{y}^{1-ip^{-n}}) \pmod{p^{n+1}}.
\end{equation}
On the other hand, since $\frac{1}{p} \binom{p^{n}}{i} \in \ZZ$ for $i=1,\dots,p^{n}-1$,
we may write
\[
z_1 - x_1 - y_1
= \frac{[\overline{x}] + [\overline{y}] - [\overline{z}]}{p} \equiv - \sum_{i=1}^{p^{n}-1}
\frac{1}{p} \binom{p^{n}}{i} [\overline{x}^{ip^{-n}} \overline{y}^{1-ip^{-n}}] \pmod{p^n},
\]
apply $T$, invoke the induction hypothesis on both sides, and multiply by $p$ to obtain
\begin{equation} \label{eq:teichmuller2}
pT(z_1) - pT(x_1) - pT(y_1)
\equiv - \sum_{i=1}^{p^{n}-1}
 \binom{p^{n}}{i} t(\overline{x}^{ip^{-n}} \overline{y}^{1-ip^{-n}}) \pmod{p^{n+1}}.
\end{equation}
Since $T(x) = T([\overline{x}]) + p T(x_1)$ and so on, we may add \eqref{eq:teichmuller1}
and \eqref{eq:teichmuller2} to deduce that $T(z) - T(x) - T(y) \equiv 0  \pmod{p^{n+1}}$,
completing the induction.
Hence $T$ is additive; it is also clear that $T$ is $p$-adically continuous. From this it follows formally that $T$ is multiplicative:
for $x = \sum_{n=0}^\infty p^n [\overline{x}_n]$,
$y = \sum_{n=0}^\infty p^n [\overline{y}_n]$,
\[
T(x) T(y) = \sum_{m,n=0}^\infty p^{m+n} t(\overline{x}_m) t(\overline{y}_n) = \sum_{m,n=0}^\infty p^{m+n} t(\overline{x}_m \overline{y}_n) = T\left( \sum_{m,n=0}^\infty p^{m+n} [\overline{x}_m \overline{y}_n] \right)= T(xy).
\]
We conclude that $T$ is a ring homomorphism as claimed.
\end{proof}

\begin{remark} \label{R:addition formula}
Take $R$ as in Example~\ref{exa:p-rings} with
$X = \{\overline{x}, \overline{y}\}$. 
By Lemma~\ref{L:Teichmuller1}, we have
\begin{equation} \label{eq:witt formulas}
[\overline{x}] - [\overline{y}] = \sum_{n=0}^\infty p^n [P_n(\overline{x}, \overline{y})]
\end{equation}
for some $P_n(\overline{x}, \overline{y})$ in the ideal $(\overline{x}^{p^{-\infty}}, \overline{y}^{p^{-\infty}})
\subset \Fp[\overline{x}^{p^{-\infty}}, \overline{y}^{p^{-\infty}}]$ 
and homogeneous of degree 1.
By Lemma~\ref{L:Teichmuller2}, \eqref{eq:witt formulas}
is also valid for any strict $p$-ring $R$ and any $\overline{x}, \overline{y} \in R/(p)$.
One can similarly derive formulas for arithmetic in a strict $p$-ring in terms
of Teichm\"uller coordinates; these can also be obtained using Witt vectors (Definition~\ref{D:Witt vectors}).
\end{remark}

\begin{theorem} \label{T:Witt}
The functor $R \rightsquigarrow R/(p)$ from strict $p$-rings to perfect rings of characteristic $p$
is an equivalence of categories.
\end{theorem}
\begin{proof}
Full faithfulness follows from Lemma~\ref{L:Teichmuller2}.
To prove essential surjectivity, 
let $\overline{R}$ be a perfect ring of characteristic $p$,
choose a surjection $\psi: \Fp[X^{p^{-\infty}}] \to \overline{R}$ for some set $X$, and put $\overline{I} = \ker(\psi)$.
Let $R_0$ be the $p$-adic completion of $\ZZ[X^{p^{-\infty}}]$; as in Example~\ref{exa:p-rings}, this is a strict $p$-ring
with $R_0/(p) \cong \Fp[X^{p^{-\infty}}]$.
Put $I = \{\sum_{n=0}^\infty p^n [\overline{x}_n] \in R_0: \overline{x}_0, \overline{x}_1,\dots \in \overline{I}\}$;
this forms an ideal in $R_0$ by Remark~\ref{R:addition formula}.
Then $R = R_0/I$ is a strict $p$-ring with $R/(p) \cong \overline{R}$.
\end{proof}

\begin{defn} \label{D:Witt vectors}
For $\overline{R}$ a perfect ring of characteristic $p$,
we write $W(\overline{R})$ for the unique (by Theorem~\ref{T:Witt})
strict $p$-ring with $W(\overline{R})/(p) \cong \overline{R}$.
This is meant as a reminder that $W(\overline{R})$ also occurs as the 
ring of \emph{$p$-typical Witt vectors} over $\overline{R}$; that construction
obtains the formulas for arithmetic in Teichm\"uller coordinates in an elegant manner linked to
symmetric functions.
\end{defn}

\begin{remark} \label{R:local ring}
Let $R$ be a strict $p$-ring. Since $R$ is $p$-adically complete, the Jacobson radical of $R$ contains $p$. In particular, if $\overline{R}$ is local, then so is $R$.
\end{remark}

\subsection{Perfect norm fields}

Theorem~\ref{T:field of norms1} is obtained by matching up the Galois correspondences of the fields
$\Fp((\overline{\pi}))$ and $\Qp(\mu_{p^\infty})$. The approach taken by Fontaine and Wintenberger is to pass from
characteristic $0$ to characteristic $p$ by looking at certain sequences of elements of finite extensions of $\Qp$
in which each term is obtained from the succeeding term by taking a certain norm between fields; the resulting functor
is thus commonly called the functor of \emph{norm fields}. It is here that some careful analysis of higher ramification
theory is needed in order to make the construction work.

While the Fontaine-Wintenberger construction gives rise directly to finite extensions of $\Fp((\overline{\pi}))$, it was later 
observed that a simpler construction (used repeatedly by Fontaine in his further study of $p$-adic Hodge theory)
could be used to obtain the perfect closures of these finite extensions. Originally the construction of these
\emph{perfect norm fields} depended crucially on the prior construction of the \emph{imperfect norm fields} of
Fontaine-Wintenberger (as in the exposition in \cite{brinon-conrad}), but we will instead work directly with
the perfect norm fields.

\begin{defn}
By an \emph{analytic field}, we will mean a field $K$ which is complete with respect to a
multiplicative nonarchimedean norm $\left|\bullet \right|$. For $K$ an analytic field, 
put $\gotho_K = \{x \in K: \left|x\right| \leq 1\}$;
this is a local ring with maximal ideal $\gothm_K = \{x \in K: \left|x\right| < 1\}$.
We say $K$ has \emph{mixed characteristics} if $\left|p\right| = p^{-1}$ (so $K$ has characteristic $0$)
and the residue field $\kappa_K = \gotho_K/\gothm_K$ of $K$ has characteristic $p$.
\end{defn}

\begin{remark} \label{R:extend analytic}
Any finite extension $L$ of an analytic field $K$ is itself an analytic field; that is,
the norm extends uniquely to a multiplicative norm on the extension field. 
A key fact about this extension is \emph{Krasner's lemma}: if $P(T) \in K[T]$ splits over $L$ as $\prod_{i=1}^n (T - \alpha_i)$
and $\beta \in K$ satisfies $\left|\alpha_1 - \beta\right| < \left|\alpha_1 - \alpha_i\right|$ for $i=2,\dots,n$, then $\alpha_1 \in K$.

A crucial consequence of Krasner's lemma 
is that an infinite algebraic extension of $K$ is separably closed if and only if its completion is algebraically closed.
More precisely, Krasner's lemma is only needed for the ``if'' implication; the ``only if'' implication follows from the fact that
the roots of a polynomial vary continuously in the coefficients. This principle also appears in the proof
of Lemma~\ref{L:alg closed case}.
\end{remark}
\begin{remark} \label{R:complete fields}
Using Krasner's lemma, one may show
that $\Qp(\mu_{p^\infty})$ and its completion have the same Galois group. Similarly,
$\Fp((\overline{\pi}))$, its perfect closure, and the completion of the perfect closure all have the same Galois group.
Consequently, from the point of view of proving Theorem~\ref{T:field of norms1},
there is no harm in considering only analytic fields.
\end{remark}

\begin{defn} \label{D:norm field}
Let $K$ be an analytic field of mixed characteristics.
Let $\gotho_{K'}$ be the inverse limit $\varprojlim \gotho_K/(p)$ under Frobenius.
In symbols,
\[
\gotho_{K'} = \left\{ (\overline{x}_n) \in \prod_{n=0}^\infty \gotho_K/(p): \overline{x}_{n+1}^p = \overline{x}_n \right\}.
\]
By construction, this is a perfect ring of characteristic $p$: the inverse of Frobenius is the shift map
$(\overline{x}_n)_{n=0}^\infty \mapsto (\overline{x}_{n+1})_{n=0}^\infty$.
By applying Lemmas~\ref{L:Teichmuller1} and~\ref{L:Teichmuller2} to the homomorphism $\overline{\theta}:
\gotho_{K'} \to \gotho_K/(p)$,
we obtain a multiplicative map $\theta: \gotho_{K'} \to \gotho_K$ and
a homomorphism $\Theta: W(\gotho_{K'}) \to \gotho_K$.

For $\overline{x} = (\overline{x}_0, \overline{x}_1,\dots) \in \gotho_{K'}$,
define $\left|\overline{x}\right|' = \left|\theta(\overline{x})\right|$.
If we lift $\overline{x}_n \in \gotho_K/(p)$ to $x_n \in \gotho_K$, then
$\left|\overline{x}\right|' = \left|x_n\right|^{p^n}$ whenever $\left|x_n\right| > \left|p\right|$. Consequently, 
$\left|\bullet\right|'$ is a multiplicative nonarchimedean norm on $\gotho_{K'}$ under which $\gotho_{K'}$ is complete
(given any Cauchy sequence, the terms in any particular position eventually stabilize).
\end{defn}

\begin{lemma} \label{L:valuation ring}
With notation as in Definition~\ref{D:norm field},
for $\overline{x}, \overline{y} \in \gotho_{K'}$, $\overline{x}$ is divisible by $\overline{y}$ if and only if 
$\left|\overline{x}\right|' \leq \left|\overline{y}\right|'$.
\end{lemma}
\begin{proof}
If $\overline{x}$ is divisible by $\overline{y}$, then
$\left|\overline{x}\right|' \leq \left|\overline{y}\right|' \left|\overline{x}/\overline{y}\right|' \leq \left|\overline{y}\right|'$.
Conversely, suppose $\left|\overline{x}\right|' \leq \left|\overline{y}\right|'$. If $\overline{y} = 0$, then $\overline{x} = 0$ also
and there is nothing more to check. Otherwise, write $\overline{x} = (\overline{x}_0, \overline{x}_1, \dots)$,
$\overline{y} = (\overline{y}_0, \overline{y}_1, \dots)$, and choose lifts $x_n, y_n$ of $\overline{x}_n, \overline{y}_n$
to $\gotho_K$. Since $\overline{y} \neq 0$, we can find an integer
$n_0 \geq 0$ such that $\left|y_n\right| \geq p^{-1+1/p}$ for $n = n_0$, and hence also for $n \geq n_0$.
Then for $n \geq n_0$, the elements $z_n = x_n/y_n \in \gotho_K$ have the property that
$\left|z_{n+1}^p - z_n\right| \leq p^{-1/p}$. By writing
$z_{n+2}^{p^2} = (z_{n+1} + (z_{n+2}^p - z_{n+1}))^p$, we deduce that
$\left|z_{n+2}^{p^2} - z_{n+1}^p\right| \leq p^{-1}$. We thus produce an element $\overline{z} = (\overline{z}_0, \overline{z}_1,\dots)$
with $\overline{x} = \overline{y} \overline{z}$ by taking $\overline{z}_n$ to be the reduction of 
$z_{n+1}^p$ for $n \geq n_0+1$.
\end{proof}

\begin{defn}
Keep notation as in Definition~\ref{D:norm field}. By Lemma~\ref{L:valuation ring},
$\gotho_{K'}$ is the valuation ring in an analytic field $K'$ which is
perfect of characteristic $p$.
We call $K'$ the \emph{perfect norm field} associated to $K$.
(Scholze \cite{scholze1} calls $K'$ the \emph{tilt} of $K$ and denotes it by $K^{\flat}$.)
\end{defn}

\begin{exercise}
The formula $x \mapsto (\theta(x^{p^{-n}}))_{n=0}^\infty$ defines a multiplicative bijection from $\gotho_{K'}$ to
the inverse limit of multiplicative monoids (but not of rings)
\[
\left\{ (x_n) \in \prod_{n=0}^\infty \gotho_K: x_{n+1}^p = x_n \right\}.
\]
This map extends to a multiplicative bijection from $K'$ to the inverse limit of multiplicative monoids
\[
\left\{ (x_n) \in \prod_{n=0}^\infty K: x_{n+1}^p = x_n \right\}.
\]
\end{exercise}

\subsection{Perfectoid fields}

In general, the perfect norm field functor is far from being faithful.
\begin{exercise} \label{exercise:compare residue fields}
Let $K$ be a discretely valued analytic field of mixed characteristics. Then
$K'$ is isomorphic to the maximal perfect subfield of $\kappa_K$.
\end{exercise}

To get around this problem, we restrict attention to analytic fields with a great deal of ramification.
(The term \emph{perfectoid} is due to Scholze \cite{scholze1}.)
\begin{defn} \label{D:deeply ramified}
An analytic field $K$ is \emph{perfectoid} if $K$ is of mixed characteristics, $K$ is not discretely valued,
and the $p$-th power Frobenius endomorphism
on $\gotho_K/(p)$ is surjective.
\end{defn}

The following statements imply that taking the perfect norm field of a perfectoid analytic field does not concede
too much information.
\begin{lemma} \label{L:perfectoid}
Let $K$ be a perfectoid analytic field.
\begin{enumerate}
\item[(a)]
We have $\left|K^\times\right| = \left|(K')^\times\right|$.
\item[(b)]
The projection $\gotho_{K'} \to \varprojlim \gotho_K/(p) \to \gotho_K/(p)$ is surjective
and induces an isomorphism $\gotho_{K'}/(\overline{z}) \cong \gotho_K/(p)$ for any 
$\overline{z} \in \gotho_{K'}$ with $\left|\overline{z}\right|' = p^{-1}$ (which exists by (a)).
In particular, we obtain a natural isomorphism $\kappa_{K'} \cong \kappa_K$.
\item[(c)]
The map $\Theta: W(\gotho_{K'}) \to \gotho_K$ is also surjective. (The kernel of $\Theta$ turns out to be a principal ideal;
see Corollary~\ref{C:theta kernel}.)
\end{enumerate}
\end{lemma}
\begin{proof}
Since $K$ is not discretely valued, we can find $r \in \left|K^\times\right|$ such that $p^{-1} r^p \in (p^{-1},1)$;
the surjectivity of Frobenius then implies $p^{-1/p} r \in \left|K^\times\right|$. This implies
$p^{-1/p} \in \left|K^\times\right|$, so $p^{-1/p} \in \left|(K')^\times\right|$ and $p^{-1} \in \left|(K')^\times\right|$.
Since also $\left|K^\times\right| \cap (p^{-1},1) = \left|(K')^\times\right| \cap (p^{-1},1)$, we obtain (a), and (b) and (c)
follow easily.
\end{proof}

\begin{cor}
The perfect norm field functor
on perfectoid analytic fields is faithful. 
\end{cor}

\begin{example} \label{exa:roots of unity}
Take $K$ to be the completion of $\Qp(\mu_{p^\infty})$ for the
unique extension of the $p$-adic norm, and fix a choice of a sequence $\{\zeta_{p^n}\}_{n=0}^\infty$
in which $\zeta_{p^n}$ is a primitive $p^n$-th root of unity and $\zeta_{p^{n+1}}^p = \zeta_{p^n}$.
The field $K$ is perfectoid because
\[
\gotho_K/(p) \cong \Fp[\overline{\zeta}_p, \overline{\zeta}_{p^2}, \dots]/(1 + \overline{\zeta}_p + \cdots +  \overline{\zeta}_p^{p-1}, \overline{\zeta}_p
- \overline{\zeta}_{p^2}^p,\dots).
\]
By the same calculation, we identify $K'$ with the completed perfect closure of $\Fp((\overline{\pi}))$
by identifying $\overline{\pi}$ with $(\overline{\zeta}_1 - 1, \overline{\zeta}_p - 1, \dots)$. This example underlies the theory
of $(\varphi, \Gamma)$-modules; see \S\ref{sec:phi-Gamma}.
\end{example}

\begin{example} \label{exa:Breuil-Kisin}
Let $F$ be a discretely valued analytic field of mixed characteristics (e.g., a finite extension of $\Qp$)
and choose a uniformizer $\pi$ of $F$. Choose a sequence $\pi_0, \pi_1, \dots$ of elements of an algebraic closure
of $F$ in which $\pi_0 = \pi$ and $\pi_{n+1}^p = \pi_n$ for $n \geq 0$. Take $K$ to be the completion of
$F(\pi_0, \pi_1, \dots)$; then 
$\gotho_K$ is the completion of $\gotho_F[\pi_1, \pi_2,\dots]$, so
\[
\gotho_K/(\pi) \cong \Fp[\overline{\pi}_1, \overline{\pi}_2, \dots]/(\overline{\pi}_1^p, \overline{\pi}_1 - \overline{\pi}_2^p, 
\dots).
\]
By Exercise~\ref{exercise:deeply ramified} below, $K$ is perfectoid. By the same calculation, we identify $K'$ with the completed perfect closure
of $\kappa_F((\overline{\pi}))$ by identifying $\overline{\pi}$ with $(\overline{\pi}_0, \overline{\pi}_1, \dots)$.
This example underlies the theory of \emph{Breuil-Kisin modules}, which are a good replacement for
$(\varphi, \Gamma)$-modules for the study of crystalline representations; see 
\cite[\S 11]{brinon-conrad}.
\end{example}

\begin{exercise} \label{exercise:deeply ramified}
Let $K$ be an analytic field of mixed characteristics which is not discretely valued.
\begin{enumerate}
\item[(a)]
Assume that there exists $\xi \in K$ with $p^{-1} \leq \left|\xi\right| < 1$ such that Frobenius is surjective on $\gotho_K/(\xi)$.
Then $K$ is perfectoid. (Hint: first imitate the
proof of Lemma~\ref{L:perfectoid}(a) to construct an element of $K$ with norm $\left| \xi \right|^{1/p}$, then construct $p$-th roots modulo successively higher powers of $\xi$.)
\item[(b)]
Suppose that there exists an ideal $I \subseteq \gothm_K$ such that the $I$-adic topology and the norm topology
on $\gotho_K$ coincide, and Frobenius is surjective on $\gotho_K/I$. Then $K$ is perfectoid.
(Note that $I$ need not be principal; for instance, take $I = \{x \in \gotho_K: \left|x\right| < p^{-1/2}\}$.
If $I$ is finitely generated, however, it is principal.)
\end{enumerate}
\end{exercise}

\begin{exercise}
An analytic field $K$ of mixed characteristics is perfectoid if and only if for every $x \in K$,
there exists $y \in K$ with $\left|y^p - x\right| \leq p^{-1} \left|x\right|$. As in the previous exercise, one can also replace
$p^{-1}$ by any constant value in the range $[p^{-1}, 1)$.
\end{exercise}

\begin{remark}
When developing the theory of norm fields, it is typical to consider the class of \emph{arithmetically profinite}
algebraic extensions of $\Qp$, i.e., those for which the Galois closure has the property that its higher ramification subgroups
are open. One then shows using Exercise~\ref{exercise:deeply ramified}
that the completions of such extensions are perfectoid.
While this construction has the useful feature of providing many examples of perfectoid fields, we will have
no further need for it here. (See however Remark~\ref{R:Tate surjective}.)
\end{remark}

\subsection{Inverting the perfect norm field functor}
\label{subsec:primitive}

So far, we have a functor from perfectoid analytic fields to perfect analytic fields of characteristic $p$.
In order to invert this functor, we must also keep track of the kernel of the map $\Theta$; this kernel turns out to 
contain elements which behave a bit like linear polynomials, in that they admit an analogue of the division
algorithm. This exploits an imperfect but useful analogy between strict $p$-rings and rings of formal power series,
in which $p$ plays the role of a series variable and the Teichm\"uller coordinates (Definition~\ref{D:Teichmuller})
play the role of coefficients. For a bit more on this analogy, see
Remark~\ref{R:Weierstrass}; for further discussion,
including a form of Weierstrass preparation in strict $p$-rings, see for instance
\cite{fargues-fontaine}.

\begin{hypothesis}
Throughout \S\ref{subsec:primitive}, let $F$ be an analytic field which is perfect of characteristic $p$. 
We denote the norm on $F$ by $\left|\bullet\right|'$.
\end{hypothesis}

\begin{remark} \label{R:norm property}
We will use frequently the following consequence of the
homogeneity aspect of Remark~\ref{R:addition formula}: the function
\[
\sum_{n=0}^\infty p^n [\overline{x}_n] \mapsto \sup_n \{\left|\overline{x}_n\right|'\}
\]
satisfies the strong triangle inequality, and hence defines a norm on $W(\gotho_F)$.
See Lemma~\ref{L:Gauss norm} for a related observation.
\end{remark}

\begin{defn} \label{D:primitive}
For $z \in W(\gotho_{F})$ with reduction $\overline{z} \in \gotho_{F}$, 
we say $z$ is \emph{primitive} if $\left|\overline{z}\right|' = p^{-1}$ and
$p^{-1}(z - [\overline{z}]) \in W(\gotho_F)^\times$. 
(The existence of such an element in particular forces the norm on $F$ not to be trivial.)
Since $W(\gotho_F)$ is a local ring by Remark~\ref{R:local ring}, for $z = \sum_{n=0}^\infty p^n [\overline{z}_n]$ it is equivalent to require that $\left|\overline{z}_0\right|' = p^{-1}$ and $\left|\overline{z}_1\right|' = 1$;
in particular, whether or not $z$ is primitive depends only on $z$ modulo $p^2$.
\end{defn}

\begin{exercise} \label{exer:primitive}
If $z \in W(\gotho_F)$ is primitive, then so is $uz$ for any $u \in W(\gotho_F)^\times$. (Hint: for $y = uz$, show that
$\left|\overline{y}_1 - \overline{u}_0 \overline{z}_1\right| < 1$ by working with Witt vectors modulo $p^2$.)
\end{exercise}

In order to state the division lemma for primitive elements, we need a slightly wider
class of elements of $W(\gotho_F)$ than the Teichm\"uller lifts.
\begin{defn}
An element $x = \sum_{n=0}^\infty p^n [\overline{x}_n] \in W(\gotho_F)$ is \emph{stable} if 
$\left|\overline{x}_n\right|' \leq \left|\overline{x}_0\right|'$ for all $n > 0$. Note that $0$ is stable under this definition.
\end{defn}

\begin{lemma} \label{L:stable product}
An element of $W(\gotho_F)$ is stable if 
and only if it equals a unit times a Teichm\"uller lift.
\end{lemma}
\begin{proof}
This is immediate from the fact that $[\overline{x}] \sum_{n=0}^\infty p^n [\overline{y}_n]
= \sum_{n=0}^\infty p^n [\overline{x} \overline{y}_n]$.
\end{proof}

Here is the desired analogue of the division lemma, taken from \cite[Lemma~5.5]{kedlaya-witt}.
\begin{lemma} \label{L:theta kernel}
For any primitive $z \in W(\gotho_F)$,
every class in $W(\gotho_F)/(z)$ is represented by a stable element of $W(\gotho_F)$.
\end{lemma}
\begin{proof}
Write $z = [\overline{z}] + p z_1$ with $z_1 \in W(\gotho_F)^\times$. 
Given $x \in W(\gotho_F)$, put $x_0 = x$.
Given $x_l = \sum_{n=0}^\infty p^n [\overline{x}_{l,n}]$ congruent to $x$ modulo $z$,
put $x_{l,1} = \sum_{n=0}^\infty p^n [\overline{x}_{l,n+1}]$
and $x_{l+1} = x_l - x_{l,1} z_1^{-1} z$, so that $x_{l+1}$ is also congruent to $x$ modulo $z$.

Suppose that for some $l$, we have
$\left|\overline{x}_{l,n}\right|' < p \left|\overline{x}_{l,0}\right|'$ for all $n>0$. 
By Remark~\ref{R:norm property} and the equality $x_{l+1} = [\overline{x}_{l,0}] - x_{l,1} z_1^{-1} [\overline{z}]$, 
we have $\left|\overline{x}_{l+1,n}\right|' \leq \left|\overline{x}_{l,0}\right|'$ 
for all $n \geq 0$. Also, $\overline{x}_{l+1,0}$ equals $\overline{x}_{l,0}$ plus something of lesser norm
(namely $\overline{z} \overline{x}_{l,1}$ times the reduction of $z_1^{-1}$), so
$\left|\overline{x}_{l+1,0}\right|' = \left|\overline{x}_{l,0}\right|'$. Hence
$x_{l+1}$ is a stable representative of the congruence class of $x$.

We may thus suppose that no such $l$ exists. By Remark~\ref{R:norm property} again,
$\sup_n \{\left|\overline{x}_{l+1,n}\right|\} \leq p^{-1} \sup_n \{\left|\overline{x}_{l,n}\right|\}$ 
for all $l$. The sum $y = \sum_{l=0}^\infty x_{l,1} z_1^{-1}$ thus converges 
for the $(p, [\overline{z}])$-adic topology on $W(\gotho_F)$ and satisfies
$x = yz$; that is, 0 is a stable representative of the congruence class of $x$.
\end{proof}

\begin{remark}
One might hope that one can always take the stable representative in Lemma~\ref{L:theta kernel}
to be a Teichm\"uller lift, but this is in general impossible unless the field $F$
is not only complete but also \emph{spherically complete}. This condition means that any decreasing
sequence of balls in $F$ has nonempty intersection; completeness only imposes this condition when
the radii of the balls tend to 0. For example, a completed algebraic closure of $\QQ_p$,
or of a power series field in characteristic $p$, is not spherically complete.
\end{remark}

\begin{lemma} \label{L:same norm}
Any stable element of $W(\gotho_F)$ divisible by a primitive element must equal $0$.
\end{lemma}
\begin{proof}
Suppose $x \in W(\gotho_F)$ is stable and is divisible by a
primitive element $z$.
Put $y = x/z$ and write $x = \sum_{n=0}^\infty p^n [\overline{x}_n]$, $y = \sum_{n=0}^\infty p^n [\overline{y}_n]$. 
Also write $z = [\overline{z}] + p z_1$ with $z_1 \in W(\gotho_F)^\times$.
Then on one hand, $\overline{x} = \overline{y}_0 \overline{z}$,
so $\left|\overline{y}_0\right|' = p \left|\overline{x}_0\right|'$.
On the other hand, by writing
\[
x = p z_1 y + [\overline{z}] y
\]
and using Remark~\ref{R:norm property}, we see that
(because the term $pz_1 y$ dominates)
\[
\sup_n \{\left|\overline{x}_n\right|\} = \sup_n \{ \left|\overline{y}_n\right| \}.
\]
Since $x$ is stable, this gives a contradiction unless $x=0$, as desired.
\end{proof}

\begin{cor} \label{C:same norm}
Suppose that $z \in W(\gotho_F)$ is primitive
and that $x,y \in W(\gotho_F)$ are stable and congruent modulo $z$.
Then the reductions of $x,y$ modulo $p$ have the same norm.
\end{cor}
\begin{proof}
Put $w = x-y$ and write 
$w =  \sum_{n=0}^\infty p^n [\overline{w}_n]$,
$x = \sum_{n=0}^\infty p^n [\overline{x}_n]$,
$y = \sum_{n=0}^\infty p^n [\overline{y}_n]$.
By Remark~\ref{R:norm property}, $\left|\overline{w}_n\right|' \leq \max\{\left|\overline{x}_0\right|', \left|\overline{y}_0\right|'\}$
for all $n \geq 0$. However, if $\left|\overline{x}_0\right|' \neq \left|\overline{y}_0\right|'$,
then $\left|\overline{w}_0\right|' = \max\{\left|\overline{x}_0\right|', \left|\overline{y}_0\right|'\} > 0$,
so $w$ is a nonzero stable element of $W(\gotho_F)$ divisible by $z$. This contradicts
Lemma~\ref{L:same norm}, so we must have $\left|\overline{x}_0\right|' = \left|\overline{y}_0\right|'$ as desired.
\end{proof}

\begin{exercise}
Give another proof of Lemma~\ref{L:same norm} by formulating a theory of Newton polygons for elements of $W(\gotho_F)$.
\end{exercise}

\begin{remark} \label{R:Weierstrass}
A good way to understand the preceding discussion is to compare it to the theory of \emph{Weierstrass preparation}
for power series over a complete discrete valuation ring. For a concrete example, consider the ideal $(T-p)$ in the ring
$\Zp\llbracket T \rrbracket$. There is a natural map $\Zp \to \Zp\llbracket T \rrbracket/(T-p)$;
one may see that this map is injective by observing that no nonzero element of $\Zp$ can be divisible by $T-p$
(by analogy with Lemma~\ref{L:same norm}), and that it is surjective by observing that one can perform the
division algorithm on power series to reduce them modulo $T-p$ to elements of $\Zp$ (by analogy with
Lemma~\ref{L:theta kernel}). 

In the situation considered here, however, we do not start with a candidate for
the quotient ring $W(\gotho_F)/(z)$. Instead, we must be a bit more careful in order to read off the properties of the quotient directly
from the division algorithm.
\end{remark}

We are now ready to invert the perfect norm field functor.
\begin{theorem} \label{T:theta kernel}
Choose any primitive $z \in W(\gotho_{F})$ and put $\gotho_K = W(\gotho_F)/(z)$. 
For $x \in \gotho_K$, apply Lemma~\ref{L:theta kernel} to find a stable
element $y = \sum_{n=0}^\infty p^n[\overline{y}_n] \in W(\gotho_F)$ lifting $x$, then define $\left|x\right| = \left|\overline{y}_0\right|'$.
This is independent of the choice of $y$ thanks to Lemma~\ref{L:same norm}.
\begin{enumerate}
\item[(a)]
The function $\left|\bullet \right|$ is a multiplicative norm on $\gotho_K$ under which $\gotho_K$ is complete.
\item[(b)]
There is a natural (in $F$) isomorphism $\gotho_K/(p) \cong \gotho_F/(\overline{z})$.
\item[(c)]
The ring $\gotho_K$ is the valuation ring of an analytic field $K$ of mixed characteristics.
\item[(d)]
The field $K$ is perfectoid and there is a natural isomorphism $K' \cong F$.
\item[(e)]
The kernel of $\Theta: W(\gotho_{F}) \to \gotho_K$ is generated by $z$.
\end{enumerate}
\end{theorem}
\begin{proof}
Part (a) follows from the fact that the product of stable elements is stable (thanks to Remark~\ref{R:norm property}).
Part (b) follows by comparing both sides to $W(\gotho_F)/(p, [\overline{z}])$.
Part (c) follows from the fact (a consequence of Lemma~\ref{L:stable product})
that if $x = \sum_{n=0}^\infty p^n [\overline{x}_n]$,
$y = \sum_{n=0}^\infty p^n [\overline{y}_n]$ are stable and
$\left|\overline{x}_0\right|' \leq \left|\overline{y}_0\right|'$, then $x$ is divisible by $y$ in $W(\gotho_F)$.
Part (d) follows from (b) plus Theorem~\ref{L:perfectoid}.
Part (e) follows from the construction of $\left|\bullet\right|$, or more precisely from the fact that
every nonzero class in $W(\gotho_F)/(z)$ has a nonzero stable representative.
\end{proof}

\begin{cor} \label{C:theta kernel}
Let $K$ be a perfectoid analytic field. Then there exists a primitive element
$z$ in the kernel of $\Theta: W(\gotho_{K'}) \to \gotho_K$, so $\ker(\Theta)$ is principal generated by $z$
by Theorem~\ref{T:theta kernel}.
(Exercise~\ref{exer:primitive} then implies that conversely, any generator of $\ker(\Theta)$ is primitive.)
\end{cor}
\begin{proof}
Since $K$ and $K'$ have the same norm group by Lemma~\ref{L:perfectoid}, we can find $\overline{z} \in \gotho_{K'}$ with $\left|\overline{z}\right|' = p^{-1}$.
Then $\theta(\overline{z})$ is divisible by $p$ in $\gotho_K$; since $\Theta$ is surjective, we can find $z_1 \in W(\gotho_{K'})$
with $\Theta(z_1) = -\theta(\overline{z})/p$. This forces $z_1 \in W(\gotho_{K'})^\times$, as otherwise
we would have $\left|\Theta(z_1)\right| < 1$. Now $z = [\overline{z}] + pz_1$ is a primitive element of $\ker(\Theta)$, as desired.
\end{proof}

\begin{example}
In Example~\ref{exa:roots of unity}, note that $\left|\overline{\pi}\right|' = p^{-p/(p-1)}$. One may then check that
\[
z = ([1+\overline{\pi}]-1)/([1+\overline{\pi}]^{1/p} - 1) = \sum_{i=0}^{p-1} [1+\overline{\pi}]^{i/p}
\]
is a primitive element of $W(\gotho_{K'})$ belonging to $\ker(\Theta)$. Hence $z$ generates the kernel
by Theorem~\ref{T:theta kernel}.
\end{example}

\begin{example}
One can also write down explicit primitive elements in some cases of Example~\ref{exa:Breuil-Kisin}.
A simple example is when $\pi = p$ (this forces the field $F$ to be absolutely unramified).
In this case,
\[
z = p  - [\overline{\pi}]
\]
is a primitive element of $W(\gotho_{K'})$ belonging to (and hence generating) $\ker(\Theta)$.
\end{example}

\subsection{Compatibility with finite extensions}

At this point, using Theorem~\ref{T:theta kernel} and Corollary~\ref{C:theta kernel}, we obtain the following statement,
which one might call the \emph{perfectoid correspondence}
(or the \emph{tilting correspondence} in the terminology of \cite{scholze1}).
\begin{theorem}[Perfectoid correspondence] \label{T:perfectoid correspondence}
The operations
\[
K \rightsquigarrow (K', \ker(\Theta: W(\gotho_{K'}) \to \gotho_K)),
\qquad
(F, I) \rightsquigarrow W(\gotho_F)[p^{-1}]/I
\]
define an equivalence of categories between perfectoid analytic fields $K$ and pairs $(F,I)$
in which $F$ is a perfect analytic field of characteristic $p$ and $I$ is an ideal of $W(\gotho_F)$
generated by a primitive element.
\end{theorem}

\begin{remark}
From Example~\ref{exa:roots of unity}
and Example~\ref{exa:Breuil-Kisin}, we see that one cannot drop the ideal $I$
in Theorem~\ref{T:theta kernel}: one can have nonisomorphic perfectoid analytic fields whose
perfect norm fields are isomorphic.
\end{remark}

We will establish that the perfectoid
correspondence is compatible with finite extensions of fields on both sides, where on the right side
we replace the ideal $I$ by its extension to the larger ring.
This will give Theorem~\ref{T:field of norms1} by taking $K$ to be the completion of $\Qp(\mu_{p^\infty})$.
However, the more general result for an arbitrary perfectoid $K$ is relevant for extending $p$-adic Hodge theory to a relative setting; see Remark~\ref{R:relative}.

The first step is to lift finite extensions from characteristic $p$; this turns out to be straightforward.

\begin{lemma} \label{L:lift finite}
Let $K$ be a perfectoid analytic field
and put $I = \ker(\Theta: W(\gotho_{K'}) \to \gotho_K)$.
Let $F$ be a finite extension of $K'$ and put $L = W(\gotho_F)[p^{-1}]/IW(\gotho_F)[p^{-1}]$.
Then $[L:K] = [F:K']$.
(Note that by Theorem~\ref{T:theta kernel}, $L$ is a perfectoid analytic field and we may identify $L'$ with $F$.)
\end{lemma}
\begin{proof}
Apply Corollary~\ref{C:theta kernel} to choose a primitive generator $z \in I$.
The extension $F/K'$ is separable because $K'$ is perfect; it thus has a 
Galois closure $\tilde{F}$. Put $\tilde{L} = W(\gotho_{\tilde{F}})[p^{-1}]/(z)$,
which is a perfectoid analytic field thanks to Theorem~\ref{T:theta kernel}.
For any subgroup $H$ of $G = \Gal(\tilde{F}/K')$,
averaging over $H$ defines a projection
\[
\tilde{L} = \frac{W(\gotho_{\tilde{F}})[p^{-1}]}{z W(\gotho_{\tilde{F}})[p^{-1}]}
\to \frac{W(\gotho_{\tilde{F}^H})[p^{-1}]}{z W(\gotho_{\tilde{F}})[p^{-1}] \cap W(\gotho_{\tilde{F}^H})[p^{-1}]}
= \frac{W(\gotho_{\tilde{F}^H})[p^{-1}]}{z W(\gotho_{\tilde{F}^H})[p^{-1}]},
\]
so the right side equals the fixed field $\tilde{L}^H$.

Put $\tilde{G} = \Gal(\tilde{F}/F)$. Apply the above analysis with $H = G$ 
and $H = \tilde{G}$; since $\tilde{F}^G = K'$ and $\tilde{F}^{\tilde{G}} = F$, 
we obtain $\tilde{L}^G = K$ and $\tilde{L}^{\tilde{G}} = L$.
By Artin's lemma, $\tilde{L}$ is a finite Galois extension of $\tilde{L}^H$ with Galois group $H$, so
\[
[L:K] = [\tilde{L}:K] / [\tilde{L}:L] = \#G/\#\tilde{G} = [\tilde{F}:K']/[\tilde{F}:F]
= [F:K'],
\]
as desired.
\end{proof}

We still need to check that a finite extension of a perfectoid analytic field is again perfectoid.
It suffices to study the case where the perfect norm field is algebraically closed.
It is this argument in particular that we learned from the course of Coleman mentioned in the introduction, and even wrote down on one previous occasion; see \cite[Theorem~7]{kedlaya-power}.
\begin{lemma} \label{L:alg closed case}
If $K$ is a perfectoid field and $K'$ is algebraically closed, then so is $K$.
\end{lemma}
\begin{proof}
Let $P(T) \in \gotho_K[T]$ be an arbitrary monic polynomial of degree $d \geq 1$; it suffices to check that
$P(T)$ has a root in $\gotho_K$.
We will achieve this by exhibiting a sequence $x_0,x_1,\dots$ of elements of $\gotho_K$ such that
for all $n \geq 0$,  $\left|P(x_n)\right| \leq p^{-n}$ and $\left|x_{n+1} - x_n\right| \leq p^{-n/d}$.
This sequence will then have a limit $x \in \gotho_K$ which is a root of $P$.

To begin, take $x_0 = 0$. Given $x_n \in \gotho_K$ with $\left|P(x_n)\right| \leq p^{-n}$,
write $P(T + x_n) = \sum_i Q_i T^i$. 
If $Q_0 = 0$, we may take $x_{n+1} = x_n$, so assume hereafter that $Q_0 \neq 0$.
Put
\[
c = \min\{\left|Q_0/Q_j\right|^{1/j}: j > 0, Q_j \neq 0\};
\]
by taking $j=d$, we see that $c \leq \left|Q_0\right|^{1/d}$.
Also, since $K$ has the same norm group as $K'$ by
Theorem~\ref{T:theta kernel},
this norm group is divisible; we thus have $c = \left|u\right|$ for some $u \in \gotho_K$.
(Note that $c$ is defined so as to equal the smallest norm of a norm of $Q$ in $\overline{K}$, but this fact is not explicitly used in the proof.)

Apply Corollary~\ref{C:theta kernel} to construct a primitive element $z \in \ker(\Theta)$.
Put $\overline{R}_0 = 0$.
For each $i>0$, choose $\overline{R}_i \in \gotho_{K'}$ 
whose image in $\gotho_{K'}/(\overline{z}) \cong \gotho_K/(p)$ is the same as that of
$Q_i u^i/Q_0$. Define the (not necessarily monic) polynomial $\overline{R}(T) = \sum_i \overline{R}_i T^i \in \gotho_{K'}[T]$.
By construction, the largest slope in the Newton polygon of $\overline{R}$ is 0; by this observation
plus the fact that $K'$ is algebraically closed, it follows that $\overline{R}(T)$ has a root
$y' \in \gotho_{K'}^\times$. Choose $y \in \gotho_K^\times$ whose image in $\gotho_K/(p) \cong \gotho_{K'}/(\overline{z})$
is the same as that of $y'$, and take $x_{n+1} = x_n + u y$.
Then $\sum_i Q_i u^i y^i / Q_0 \equiv 0 \pmod{p}$, so
$\left|P(x_{n+1})\right| \leq p^{-1} \left|Q_0\right| \leq p^{-n-1}$
and $\left|x_{n+1} - x_n\right| = \left|u\right| \leq \left|Q_0\right|^{1/d} \leq p^{-n/d}$.
We thus obtain the desired sequence, proving the claim.
\end{proof}

\begin{remark}
The inertia subgroup of the Galois group of a finite extension of analytic fields is solvable;
see for instance \cite[Chapter~3]{kedlaya-course} and references therein. (The discretely valued case may be 
more familiar; for that, see also \cite[Chapter~IV]{serre-local-fields}.) Using this statement,
one can give a slightly simpler proof of Lemma~\ref{L:alg closed case} by considering only
cyclic extensions of prime degree. We chose not to proceed this way so as to make good on our promise to keep the proof of
Theorem~\ref{T:field of norms1} entirely free of ramification theory.
\end{remark}

We  are now ready to complete the proof
of Theorem~\ref{T:field of norms1}.

\begin{theorem} \label{T:field of norms general}
Let $K$ be a perfectoid analytic field. Then
every finite extension of $K$ is perfectoid, and the operation $L \rightsquigarrow L'$
defines a functorial correspondence between the finite extensions of $K$ and $K'$. In particular,
the absolute Galois groups of $K$ and $K'$ are homeomorphic.
\end{theorem}
\begin{proof}
Apply Corollary~\ref{C:theta kernel} to construct a primitive generator $z \in \ker(\Theta: W(\gotho_{K'}) \to \gotho_K)$.
Let $M'$ be the completion of an algebraic closure $\overline{K}'$ of $K'$; it is again
algebraically closed by Remark~\ref{R:extend analytic}.
By Theorem~\ref{T:perfectoid correspondence},
$M'$ arises as the perfect norm field of a perfectoid analytic field $M$,
which by Lemma~\ref{L:alg closed case} is also algebraically closed.

By Lemma~\ref{L:lift finite}, each finite Galois extension $L'$ of $K'$ within $M'$
is the perfect norm field of a finite Galois extension $L$ of $K$ within $M$ which is perfectoid.
The union $\tilde{L}$ of such fields $L$ is dense in $M$ because the union of the $L'$ is dense in $M'$.
By Remark~\ref{R:extend analytic}, $\tilde{L}$ is algebraically closed; that is, every finite
extension of $K$ is contained in a finite Galois extension which is perfectoid.
The rest follows from Theorem~\ref{T:perfectoid correspondence}.
\end{proof}

\begin{remark} \label{R:scholze}
The proof of Theorem~\ref{T:field of norms general} is a digested version of the one given in
\cite[Theorem~3.5.6]{kedlaya-liu1}.
A different proof has been given by Scholze \cite[Theorem~3.7]{scholze1}, in which the analysis of strict $p$-rings is supplanted
by use of a small amount of \emph{almost ring theory}, as introduced by Faltings and developed systematically
by Gabber and Ramero \cite{gabber-ramero}. However, these two approaches
resemble each other far more strongly than either one resembles the original arguments of Fontaine and Wintenberger.
\end{remark}

\begin{remark} \label{R:relative}
In both \cite{kedlaya-liu1} and \cite{scholze1}, Theorem~\ref{T:field of norms general} is generalized to a statement
relating the \'etale sites of certain nonarchimedean analytic spaces in characteristic 0 and characteristic $p$,
including an optimally general form of Faltings's \emph{almost purity theorem}. (For the flavor
of this result, see Theorem~\ref{T:formally unramified}.)
This is used as a basis
for relative $p$-adic Hodge theory in \cite{kedlaya-liu1, kedlaya-liu2} and \cite{scholze2}. Note that for this application,
it is crucial to have Theorem~\ref{T:field of norms general} and not just
Theorem~\ref{T:field of norms1}: one must use analytic spaces in the sense of Huber (\emph{adic spaces})
rather than rigid analytic spaces, which forces an encounter with general analytic fields. One must also deal with valuations of rank greater than 1 (i.e., whose value groups do not fit inside the real numbers), but this adds no essential difficulty.
\end{remark}

The following is taken from \cite[Proposition~3.5.9]{kedlaya-liu1}.
\begin{exercise}
Let $L/K$ be a finite extension of analytic fields such that $L$ is perfectoid. Prove that $K$ is also perfectoid. Hint: reduce to the Galois case, then produce a Galois-invariant primitive element.
\end{exercise}

\subsection{Some applications}

We describe now a couple of applications of Theorem~\ref{T:field of norms general},
in which one derives information in characteristic $0$ by exploiting Frobenius
as if one were working in positive characteristic.

\begin{defn}
An analytic field $K$ is \emph{deeply ramified} if for any finite extension $L$ of $K$, 
$\Omega_{\gotho_L/\gotho_K} = 0$; that is, the morphism $\Spec(\gotho_L) \to \Spec(\gotho_K)$ is formally unramified.
(Beware that this morphism is usually not of finite type if $K$ is not discretely valued.)
\end{defn}

\begin{theorem} \label{T:formally unramified}
Any perfectoid analytic field is deeply ramified.
\end{theorem}
\begin{proof}
Let $K$ be a perfectoid field and let $L$ be a finite extension of $K$.
Choose $x_1,\dots,x_n \in \gotho_L$ which form a basis of $L$ over $K$;
we can then find $t \in \gotho_K - \{0\}$ such that $t\gotho_L \subseteq \gotho_K x_1 + \cdots + \gotho_K x_n$.
Since $L$ is a finite separable extension of $K$, $\Omega_{L/K} = 0$; consequently, we can choose $u \in \gotho_K - \{0\}$
so that $u\,dx_i$ vanishes in $\Omega_{\gotho_L/\gotho_K}$ for each $i$.
For any $x \in \gotho_L$, $t\,dx = d(tx)$ is a $\gotho_K$-linear combination of $x_1,\dots,x_n$, so
$tu\,dx = 0$.

On the other hand, $L$ is perfectoid by Theorem~\ref{T:field of norms general}.
Hence for any $x \in \gotho_L$, we can find $y \in \gotho_L$ for which $x \equiv y^p \pmod{p}$;
this implies that $\Omega_{\gotho_L/\gotho_K} = p \Omega_{\gotho_L/\gotho_K}$.
As a result, $\Omega_{\gotho_L/\gotho_K} = p^n \Omega_{\gotho_L/\gotho_K}$ for any positive integer $n$;
by choosing $n$ large enough that $tu$ is divisible by $p^n$, we deduce that $\Omega_{\gotho_L/\gotho_K} = 0$.
Hence $K$ is deeply ramified.
\end{proof}

\begin{remark}
Theorem~\ref{T:formally unramified} admits the following converse:
any analytic field of mixed characteristics which is deeply ramified is also perfectoid.
See \cite[Proposition~6.6.6]{gabber-ramero}.
\end{remark}

\begin{theorem} \label{T:Tate surjective}
Let $K$ be a perfectoid analytic field and let $L$ be a finite extension of $K$. Then
$\Trace: \gothm_L \to \gothm_K$ is surjective.
\end{theorem}
\begin{proof}
By Theorem~\ref{T:field of norms general}, $L$ is also perfectoid. Let $K', L'$ be the perfect norm fields
of $K,L$. 
Since $L'$ is a finite separable extension of $K'$, there exists $\overline{u} \in \gothm_{K'}$
such that $\overline{u} \gotho_{K'} \subseteq \Trace(\gothm_{L'})$. By applying the inverse
of Frobenius, we obtain the same conclusion with $\overline{u}$ replaced by 
$\overline{u}^{p^{-n}}$ for each positive integer $n$. Hence 
$\Trace: \gothm_{L'} \to \gothm_{K'}$ is surjective.

Since $K$ is not discretely valued,
we can find $t \in \gothm_K$ with $p^{-1} < \left|t\right| < 1$.
Since $L$ is a finite separable extension of $K$, there exists a nonnegative integer $m$ such that
$(p/t)^m \gothm_K \subseteq \Trace(\gothm_L)$.
If $m>0$, then for each $x \in (p/t)^{m-1} \gothm_K$, by the previous paragraph we may write
$x = \Trace(y) + pz$ for some $y \in \gothm_L$, $z \in \gotho_K$; since
$pz = (p/t)(tz) \in (p/t)^m \gothm_K$, it follows that $z \in \Trace(\gothm_L)$ and hence $x \in \Trace(\gothm_L)$.
In other words, we may replace $m$ by $m-1$; this proves the desired result.
\end{proof}

\begin{remark} \label{R:Tate surjective}
Note that Theorem~\ref{T:Tate surjective} still holds, with the same proof, if we replace the perfectoid field
$K$ by a dense subfield as long as the norm extends uniquely to the finite extension $L$ of $K$ (i.e., $\gotho_K$ is \emph{henselian} in the sense of Lemma~\ref{L:henselian} below).
For instance, we may take $K$ to be an infinite algebraic extension of an analytic field $F$ of mixed characteristics.

One important case is when $F = \Qp$ and $K$ is a Galois extension whose Galois group contains $\Zp$ (e.g., any
$p$-adic Lie group). In this case, the perfectoid property can be checked using a study of higher ramification
groups; the technique was introduced by Tate and developed further by Sen \cite{sen-lie} (see also
\cite[\S 13]{brinon-conrad}).

However, no ramification theory is necessary in case
$K$ is a field for which the perfectoid property can be checked directly
(as in Example~\ref{exa:roots of unity} or Example~\ref{exa:Breuil-Kisin}), or a finite extension of such 
a field (using Theorem~\ref{T:field of norms general}).
\end{remark}

\subsection{Gauss norms}
\label{subsec:Gauss}

We conclude this section by appending some more observations about norms on strict $p$-rings,
for use in  the second half of the paper.
\begin{hypothesis}
Throughout \S\ref{subsec:Gauss}, again let $F$ be an analytic field which is perfect of characteristic $p$,
with norm $\left|\bullet \right|'$.
\end{hypothesis}

The key construction here is an analogue of the Gauss norm, following \cite[Lemma~4.1]{kedlaya-witt}.

\begin{lemma} \label{L:Gauss norm}
For $r > 0$, the formula
\begin{equation} \label{eq:tilted norm}
\left| \sum_{n=0}^\infty p^n [\overline{x}_n] \right|_r =
\sup_n \{p^{-n} (\left|\overline{x}_n\right|')^r \}
\end{equation}
defines a function $\left|\bullet \right|_r: W(F) \to [0, +\infty]$ satisfying the strong triangle inequality
$\left|x+y\right|_r \leq \max\{\left|x\right|_r, \left|y\right|_r\}$ and the multiplicativity property
$\left|xy\right|_r = \left|x\right|_r \left|y\right|_r$ for all $x,y \in W(F)$ (under the convention $0 \times +\infty = 0$).
In particular, the subset $W^{r-}(F)$ of $W(F)$ on which $\left|\bullet \right|_r$ is finite forms a ring
on which $\left| \bullet \right|_r$ defines a multiplicative nonarchimedean norm.
\end{lemma}
\begin{proof}
From Remark~\ref{R:addition formula}, it follows that $\left|[\overline{x}] + [\overline{y}]\right|_r \leq
\max\{\left|[\overline{x}]\right|_r, \left|[\overline{y}]\right|_r\}$ for all $\overline{x}, \overline{y} \in \gotho_F$
(as in Remark~\ref{R:norm property}).
It follows easily that for $x,y \in F$, $\left|x+y\right|_r \leq \max\{\left|x\right|_r, \left|y\right|_r\}$ 
and $\left|xy\right|_r \leq \left|x\right|_r \left|y\right|_r$.
To establish multiplicativity, by continuity it is enough to consider finite sums
$x = \sum_{m=0}^M p^m [\overline{x}_m], y = \sum_{n=0}^N p^N [\overline{y}_n]$. For all but finitely many
$r>0$, the quantities
\[
p^{-m-n} (\left|\overline{x}_m \overline{y}_n\right|')^r \qquad (m=0,\dots,M; n=0,\dots,N)
\]
are all distinct; for such $r$, the fact that $\left| \bullet \right|_r$ is a nonarchimedean norm (shown above) gives
\[
\left|xy\right|_r = \max\{p^{-m-n} (\left|\overline{x}_m \overline{y}_n\right|')^r:  m=0,\dots,M; n=0,\dots,N\}
= \left|x\right|_r \left|y\right|_r.
\]
Since $\left|x \right|_r, \left|y\right|_r, \left|xy\right|_r$ are all continuous functions of $r$, we may infer multiplicativity also in the exceptional cases.
\end{proof}

\begin{remark} \label{R:Hadamard}
For $s \in (0,r]$ and $x = \sum_{n=0}^\infty p^n [\overline{x}_n]$,
\begin{equation} \label{eq:convex1}
\left|x\right|_s = \sup_n \{p^{-n} (\left|\overline{x}_n\right|')^s \}
\leq \sup_n \{p^{-ns/r} (\left|\overline{x}_n\right|')^s \} = \left|x\right|_r^{s/r}.
\end{equation}
An even stronger statement is that for $x \in W(F)$ fixed,
the function $\log(\left|\bullet \right|_r): W(F) \to \mathbb{R} \cup \{+\infty\}$ is convex
(because it is the supremum of convex functions); this bears a certain formal resemblance
to the \emph{Hadamard three circles theorem} in complex analysis.
A related observation is that 
\begin{equation} \label{eq:limsup}
\begin{cases}
\lim_{r \to 0^+} \left| x \right|_r = 1 & (\overline{x}_0 \neq 0) \\
\limsup_{r \to 0^+} \left| x \right|_r \leq p^{-1} & (\overline{x}_0 = 0).
\end{cases}
\end{equation}
\end{remark}

\begin{lemma} \label{L:completeness}
For $r>0$, the ring $W^{r-}(F)$ is complete with respect to $\left| \bullet \right|_r$.
\end{lemma}
\begin{proof}
For $x = \sum_{n=0}^\infty p^n [\overline{x}_n]$,
$y = \sum_{n=0}^\infty p^n [\overline{y}_n]$ with 
$\left| x \right|, \left| y \right| \leq c$,
$\left| x-y \right| \leq \epsilon c$ for some $c>0$, $\epsilon \in [0,1)$, Remark~\ref{R:addition formula} implies that
\begin{equation} \label{eq:completeness}
\left| \overline{x}_n - \overline{y}_n \right|
\leq c p^{n/r} \max\{ p^{-m/r} \epsilon^{1/p^{m}}
 : m =0,\dots,n \}.
\end{equation}
Let $x_1, x_2,\dots$ be a Cauchy sequence in $W^{r-}(F)$, and write
$x_i = \sum_{n=0}^\infty p^n [\overline{x}_{i,n}]$.
Choose $c$ such that $\left| x_i \right|_r \leq c$ for all $i$.
Using \eqref{eq:completeness}, we see that for each $n$, the sequence $\overline{x}_{i,n}$
is Cauchy and hence converges to a limit $\overline{x}_n \in F$. Put $x = \sum_{n=0}^\infty p^n [\overline{x}_n]$; we see easily that 
$\left| x \right|_r \leq c$ and hence $x \in W^{r-}(F)$. It thus remains to check that $x$ is a limit of $x_1,x_2,\dots$; this formally reduces to the case $x=0$. 
For each $\delta > 0$, by definition there exists $N>0$ such that $\left| x_i - x_j \right|_r \leq c \delta$ for all $i,j \geq N$. For each $n$, by \eqref{eq:completeness}
we have
\[
\left| \overline{x}_{i,n} - \overline{x}_{j,n} \right| 
\leq c^{1/r} p^{n/r} \max\{ p^{-m/r} \delta^{1/(p^{m}r)}
 : m =0,\dots,n \};
\]
since the sequence $\overline{x}_{i,n}$ converges to 0 as $i \to \infty$, we must also have
\[
\left| \overline{x}_{i,n} \right| 
\leq c^{1/r} p^{n/r} \max\{ p^{-m/r} \delta^{1/(p^{m}r)}
 : m =0,\dots,n \}
\]
and so
\[
\left| x_i \right|_r
\leq c \max\{ p^{-m} \delta^{1/p^{m}}: m =0,1,\dots\}.
\]
As $\delta \to 0$, this upper bound tends to 0; this proves that $x_i \to 0$ as desired.
\end{proof}

\begin{lemma}\label{L:henselian}
The ring $W^\dagger(F) = \cup_{r>0} W^{r-}(F)$ is a local ring which is not complete
but is \emph{henselian}; that is, every finite \'etale $F$-algebra lifts uniquely
to a finite \'etale $W^\dagger(F)$-algebra.
\end{lemma}
\begin{proof}
We first check that $W^\dagger(F)$ is local; it suffices to check that $p$ belongs to the Jacobson radical, meaning that any $x = \sum_{n=0}^\infty p^n [\overline{x}_n] \in W^\dagger(F)$ with $\overline{x}_0 = 1$ is a unit. By \eqref{eq:limsup},
for $r>0$ sufficiently small we have $\left| x-1 \right|_r < 1$.
By Lemma~\ref{L:completeness}, the geometric series $\sum_{i=0}^\infty (1-x)^i$ converges in $W^{r-}(F)$ to a multiplicative inverse of $x$.

We next check that $W^\dagger(F)$ is henselian. As described in \cite[Remark~1.2.7]{kedlaya-liu1}
(see also \cite[Theorem~5.11]{gabber-ramero}
and \cite[Expos\'e XI, \S 2]{raynaud}), to check that $W^\dagger(F)$ is henselian,
it suffices to check a special form of Hensel's lemma: every monic polynomial $P(T) = \sum_i P_i T^i \in W^\dagger(F)[T]$ with $P_0 \in p W^\dagger(F)$, $P_1 \notin p W^\dagger(F)$ has a root in $p W^\dagger(F)$.
By the previous paragraph, $P_1 \in W^\dagger(F)^\times$. By \eqref{eq:limsup},
for $r>0$ sufficiently small,
\[
\left| P_0/P_1 \right|_r < 1, \quad
\left| P_i (P_0/P_1)^{i} \right|_r < \left| P_1 \right|^2_r, \quad
\left| P_i (P_0/P_1)^{i-1} \right|_r < \left| P_1 \right|_r
 \qquad (i \geq 2).
\]
In particular, for $x_0 = -P_0/P_1$, for $r>0$ sufficiently small we have
\[
\left| P'(x_0) \right|_r < \left| P(x_0) \right|_r^2,
\]
so the usual Newton-Raphson iteration
\[
x_{n+1} = x_n - \frac{P(x_n)}{P'(x_n)}
\]
converges to a root $x$ of $P(T)$ in $W^{r-}(F)$. Since the iteration also converges $p$-adically, we must also have $x \in p W(F)$; this proves the claim.
\end{proof}

\begin{exercise} \label{exer:smaller ring}
We describe here a common variant of $W^{r-}(F)$. 
\begin{enumerate}
\item[(a)]
For any $\overline{z} \in \gotho_F$ with $\left|\overline{z}\right|' < 1$,
the ring $W(\gotho_F)[[\overline{z}]^{-1}]$ consists of those elements of $W(F)$ with bounded Teichm\"uller coordinates.
\item[(b)]
The completion of $W(\gotho_F)[[\overline{z}]^{-1}]$ under $\left| \bullet \right|_r$ consists of those
$x = \sum_{n=0}^\infty p^n [\overline{x}_n]$ for which $\lim_{n \to \infty} p^{-n} (\left|\overline{x}_n\right|')^r = 0$.
In particular, these form a ring, denoted $W^r(F)$.
\item[(c)]
The ring $W^r(F)$ is contained in $W^{r-}(F)$ and contains $W^{s-}(F)$ for all $s > r$.
Consequently, we also have $W^\dagger(F) = \cup_{r>0} W^r(F)$.
\item[(d)]
The ring $W^r(F)$ is complete with respect to $\left| \bullet \right|_r$. 
\end{enumerate}
The rings $W^r(F)$ arise naturally in the geometric interpretation of $p$-adic Hodge theory given by Fargues and Fontaine
\cite{fargues-fontaine}; this interpretation is facilitated by some useful noetherian properties of these rings, as described in \cite{kedlaya-noetherian}.
\end{exercise}

\begin{remark}
Let $\varphi$ denote the endomorphism of $W(F)$ induced by the Frobenius map on $F$.
For any $r>0$,
using the fact that $W^\dagger(F) = \cup_{n=0}^\infty W^{p^{-n} r-}(F) = \cup_{n=0}^\infty \varphi^{-n}(W^{r-}(F))$, it can be shown that
every finite \'etale $F$-algebra lifts uniquely to a finite \'etale $W^{r-}(F)$-algebra.
In case $F = K'$ for some deeply ramified analytic field $K$ of mixed characteristics, the
map $\Theta: W(\gotho_F) \to \gotho_K$ extends to a homomorphism $W^{r-}(F) \to K$ for any $r >1$, and 
Theorem~\ref{T:field of norms general} is equivalent to the statement that every finite \'etale
$K$-algebra lifts uniquely to $W^{r-}(F)$. (If we replace $W^{r-}(F)$ with the ring $W^r(F)$ of Exercise~\ref{exer:smaller ring}, we may also take $r=1$ in this last statement.)
\end{remark}

\section{Galois representations and $(\varphi, \Gamma)$-modules}
\label{sec:phi-Gamma}

\setcounter{theorem}{0}
\begin{hypothesis}
Throughout \S\ref{sec:phi-Gamma}, let $K$ denote a finite extension of $\Qp$, 
and write $G_K$ for the absolute Galois group of $K$.
\end{hypothesis}

\begin{convention}
When working with a matrix $A$ over a ring carrying a norm $\left| \bullet \right|$, we will write $\left|A\right|$ for the supremum of the
norms of the entries of $A$ (rather than the operator norm or spectral radius).
\end{convention}

\begin{defn}
Let $\FEt(R)$ denote the category of finite \'etale algebras over a ring $R$. For
$R$ a field,
these are just the direct sums of finite separable field extensions of $R$.
\end{defn}

\subsection{Some period rings over $\Qp$}

We will consider four different rings which can classify $G_K$-representations. We first introduce them all in the case
$K = \Qp$.

\begin{defn} \label{D:weak topology}
Let $L$ be the completed perfect closure of $\Fp((\overline{\pi}))$,
and put $\tilde{\bA}_{\Qp} = W(L)$. This defines a complete topological ring both for the $p$-adic topology
and for the \emph{weak topology}, under which a sequence converges if each sequence of Teichm\"uller coordinates
converges in the norm topology on $L$. 
(The restriction of the weak topology to $W(\gotho_L)$ coincides with the $(p, [\overline{\pi}])$-adic topology.)
Let $\varphi$ be the endomorphism of $\tilde{\bA}_{\Qp}$ induced by the Frobenius
map on $L$.
\end{defn}

\begin{defn} \label{D:imperfect}
Let $\bA_{\Qp}$ be the $p$-adic completion of $\ZZ((\pi))$; it is a Cohen ring (a complete discrete valuation ring
with maximal ideal $(p)$) with
$\bA_{\Qp}/(p) \cong \Fp((\overline{\pi}))$. 
We identify $\bA_{\Qp}$ with a subring of $\tilde{\bA}_{\Qp}$ in such a way that $\pi$
corresponds to $[1 + \overline{\pi}] - 1$. Note that $\varphi$ then acts on $\bA_{\Qp}$ as the
$\Zp$-linear substitution $\pi \mapsto (1 + \pi)^p - 1$,
and a sequence in $\bA_{\Qp}$ converges for the weak topology if and only if its image in $\bA_{\Qp}/(p^n) \cong (\ZZ/p^n \ZZ)((\pi))$ converges
$\pi$-adically for each positive integer $n$.
\end{defn}

\begin{defn}
Put $\tilde{\bA}^\dagger_{\Qp} = W^\dagger(L)$;
by Lemma~\ref{L:henselian}, this is an incomplete but henselian local ring
contained in $W(L) = \tilde{\bA}_{\Qp}$. Note that $\varphi$ acts bijectively on $\tilde{\bA}^{\dagger}_{\Qp}$.
We equip $\tilde{\bA}^{\dagger}_{\Qp}$ with the $p$-adic and weak topologies by restriction from $\tilde{\bA}_{\Qp}$;
we also define the \emph{LF topology}, in which a sequence converges if and only if it converges in some $W^{r-}(L)$.
(LF is an abbreviation for \emph{limit of Fr\'echet}.)
\end{defn}

\begin{defn} \label{D:dagger}
Put $\bA^{\dagger}_{\Qp} = \tilde{\bA}^\dagger_{\Qp} \cap \bA_{\Qp}$;
since $\Zp[\pi^{\pm}] \subset \bA^{\dagger}_{\Qp}$,
$\bA^{\dagger}_{\Qp}$ is again a henselian local ring with residue field $\Fp((\overline{\pi}))$
on which $\varphi$ acts.
It inherits $p$-adic, weak, and LF topologies.
For a more concrete description of $\bA^{\dagger}_{\Qp}$, see Corollary~\ref{C:describe dagger rings}.
\end{defn}

\begin{defn} \label{D:gamma}
For $\gamma \in \Gamma = \Zp^\times$, let $\gamma: \bA_{\Qp} \to \bA_{\Qp}$ be the $\Zp$-linear substitution
$\pi \mapsto (1 + \pi)^\gamma - 1$, where $(1 + \pi)^\gamma$ is defined via its binomial expansion.
The induced map on $\Fp((\overline{\pi}))$ extends to $L$ and thus defines an action of $\Gamma$
on $\tilde{\bA}_{\Qp}$; $\Gamma$ also acts on $\tilde{\bA}^{\dagger}_{\Qp}$ and $\bA^\dagger_{\Qp}$.
For $* \in \{\tilde{\bA}, \bA, \tilde{\bA}^\dagger, \bA^\dagger\}$,
the action of $\Gamma$ on $*_{\Qp}$ is continuous 
(meaning that the action map $\Gamma \times *_{\Qp} \to *_{\Qp}$ is continuous) for the weak topology
and (when available) the LF topology.
\end{defn}

\begin{exercise}
In Definition~\ref{D:gamma},
the action of $\Gamma$ on $*_{\Qp}$ is not continuous for the $p$-adic topology, even though the action of each individual element of $\Gamma$ is a continuous map from $*_{\Qp}$ to itself.
\end{exercise}

\subsection{Extensions of $\Qp$}

We extend the definition of the period rings to finite extensions of $\Qp$ using a refinement of
Theorem~\ref{T:field of norms general}.

\begin{theorem} \label{T:field of norms specific}
For $* \in \{ \tilde{\bA}, \bA, \tilde{\bA}^{\dagger}, \bA^\dagger\}$, the category $\FEt(\Qp)$
is equivalent to the category of finite \'etale algebras over $*_{\Qp}$ admitting an extension of the action of
$\Gamma$. Moreover, this equivalence is compatible with the base extensions among different choices of $*$.
\end{theorem}
\begin{proof}
Put $L = \tilde{\bA}_{\Qp}/(p)$.
Via Remark~\ref{R:complete fields}, Theorem~\ref{T:field of norms general}, and the fact that the local rings $\tilde{\bA}^{\dagger}_{\Qp}$ and $\bA^\dagger_{\Qp}$ are both henselian, we see that the categories
\[
\FEt(\Qp(\mu_{p^\infty})), \FEt(L), \FEt(\Fp((\overline{\pi}))), \FEt(\tilde{\bA}_{\Qp}), \FEt(\bA_{\Qp}), \FEt(\tilde{\bA}^\dagger_{\Qp}), \FEt(\bA^\dagger_{\Qp})
\]
are all equivalent, compatibly with base extensions among different choices of $*$. It thus suffices to consider
$* = \tilde{\bA}$ in what follows.

From the explicit description given in Example~\ref{exa:roots of unity}, we see that the
map $\Theta: W(\gotho_L) \to \Zp[\mu_{p^\infty}]$ becomes $\Gamma$-equivariant if we
identify $\Gamma$ with $\Gal(\Qp(\mu_{p^\infty})/\Qp)$ via the cyclotomic character. 
Consequently, for $K \in \FEt(\Qp)$, the object in $\FEt(\tilde{\bA}_{\Qp})$ corresponding to
$K \otimes_{\Qp} \Qp(\mu_{p^\infty})$ carries an action of $\Gamma$.

Conversely, suppose $S \in \FEt(\tilde{\bA}_{\Qp})$ carries an action of $\Gamma$;
then the corresponding object $E$ of $\FEt(\Qp(\mu_{p^\infty}))$ also carries an action of 
$\Gamma$. We may realize $E$ as the base extension of a finite \'etale algebra $E_n$ over $\Qp(\mu_{p^n})$
for some nonnegative integer $n$; by Artin's lemma, $E_n$ is fixed by a subgroup of $\Gamma$ of finite index,
which is necessarily open. By Galois descent, $E_n$ descends to a finite \'etale algebra $E$ over $\Qp$,
as desired.
\end{proof}

\begin{defn}
Let $\tilde{\bA}_K, \bA_{K}, \tilde{\bA}^\dagger_{K}, \bA^\dagger_{K}$
be the objects corresponding to $K$ via Theorem~\ref{T:field of norms specific}.
We may write $\tilde{\bA}_K = \oplus W(\tilde{L}), \tilde{\bA}^\dagger_K = \oplus W^\dagger(\tilde{L})$ for $\tilde{L}$ running over the connected components of $\tilde{\bA}_K/(p)$ (which correspond to the connected components
of $K \otimes_{\Qp} \Qp(\mu_{p^\infty})$). We may thus equip $*_K$ with a $p$-adic topology, a weak topology, and
(for $* = \tilde{\bA}^{\dagger}, \bA^\dagger$) also an LF topology.
Define the norm $\left| \bullet \right|'$ on $\tilde{\bA}_K/(p)$ as the supremum over connected components; for $r>0$,
define $\left| \bullet \right|_r$ on $\tilde{\bA}_K$ as the supremum over connected components.

The actions of $\varphi, \Gamma$ extend to $*_K$; the action of $\Gamma$ is again continuous for the weak topology
and (when available) the LF topology. Note that $\Gamma$ acts transitively on the $\tilde{L}$.
\end{defn}

\begin{exercise}
Each of the topologies on $*_K$ coincides with the one obtained by viewing $*_K$ as a finite free module
over $*_{\Qp}$ and equipping the latter with the corresponding topology.
\end{exercise}

\begin{remark} \label{R:power series}
Each connected component $L$ of $\bA_K/(p)$ is a finite separable extension of $\Fp((\overline{\pi}))$, and hence is itself
isomorphic to a power series field in some variable $\overline{\pi}_L$ over some finite extension $\mathbb{F}_q$ of $\Fp$.
In general, there is no distinguished choice of $\overline{\pi}_L$.
One has similar (and similarly undistinguished) descriptions of $\bA_K$ and $\bA^\dagger_K$; see
Lemma~\ref{L:describe rings}.
\end{remark}

\begin{exercise}
In Remark~\ref{R:power series}, $\mathbb{F}_q$ coincides with the residue field of $K(\mu_{p^\infty})$. Note that
this may not equal the residue field of $K$.
\end{exercise}

\begin{lemma} \label{L:describe rings}
Keep notation as in Remark~\ref{R:power series}.
Let $R$ be the connected component of $\bA_K$ with $R/(p) = L$, and choose $\pi_L \in R$ lifting $\overline{\pi}_L$.
Then $R$ is isomorphic to the $p$-adic completion of $W(\mathbb{F}_q)((\pi_L))$.
\end{lemma}
\begin{proof}
It suffices to observe that the latter ring is indeed a finite \'etale algebra over $\bA_{\Qp}$ whose reduction modulo
$p$ is isomorphic to $L$.
\end{proof}

\begin{exercise}
In Lemma~\ref{L:describe rings}, the weak topology on $R$ (obtained by restriction from $\bA_K$) coincides with the
weak topology on the $p$-adic completion of $W(\mathbb{F}_q)((\pi_L))$ (in which as in Definition~\ref{D:imperfect},
a sequence converges if and only if it converges $\pi_L$-adically modulo each power of $p$).
\end{exercise}

\begin{lemma} \label{L:describe dagger rings}
Keep notation as in Lemma~\ref{L:describe rings}, but assume further that
$\pi_L \in R^\dagger$ for $R^\dagger$ the connected component of $\bA^\dagger_K$ with $R^\dagger/(p) = L$.
Then there exists $r_0 > 0$ (depending on $L$ and $\pi_L$) with the following properties.
\begin{enumerate}
\item[(a)]
Every $\overline{x} \in R/(p)$ admits a lift
$x \in R^\dagger$ with $\left|x - [\overline{x}]\right|_{r} < \left|x\right|_{r} < +\infty$ for all $r \in (0, r_0]$.
\item[(b)]
For $r \in (0,r_0]$, for $x =\sum_{n \in \ZZ} x_n \pi_L^n \in R$ with $x_n \in 
W(\mathbb{F}_q)$,
\begin{equation} \label{eq:compare norm}
\left|x\right|_r = \sup_n \{\left|x_n\right| (\left|\overline{\pi}_L\right|')^{nr}\}.
\end{equation}
\end{enumerate}
\end{lemma}
\begin{proof}
We have $R^\dagger = R \cap W^\dagger(L)$ because both sides are subrings of $R$ which are finite \'etale over
$\bA^\dagger_{\Qp}$ and which surject onto $R/(p)$. Consequently, $\pi_L \in W^{r-}(L)$ for some $r>0$.
By \eqref{eq:limsup},
we can choose $r_0>0$ so that 
\begin{equation} \label{eq:bound pi}
\left|\pi_L - [\overline{\pi}_{L}]\right|_{r} < \left|\pi_L\right|_{r} = \left|[\overline{\pi}_{L}]\right|_{r}
\qquad (r \in (0,r_0]).
\end{equation}
We prove the claims for any such $r_0$.

To prove (a), lift $\overline{x} = \sum_{n \in \ZZ}\overline{x}_n \overline{\pi}_L^n \in R/(p)$ to $x = \sum_{n \in \ZZ} [\overline{x}_n] \pi_L^n \in R$. This lift
satisfies $\left|x - [\overline{x}]\right|_r < (\left|\overline{x}\right|')^r = \left|[\overline{x}]\right|_r = \left|x\right|_r$ for all $r \in (0, r_0]$ thanks to
\eqref{eq:bound pi}.

To prove (b), first note that since $\left|\bullet \right|_r$ is a norm, \eqref{eq:bound pi}
implies 
\[
\left|x\right|_r  \leq \sup_n \{\left|x_n\right| (\left|\overline{\pi}_L\right|')^{nr}\}.
\]
To finish, it is enough to establish by induction that for each nonnegative integer $m$,
$\left|x\right|_r$ is at least the supremum of $\left|x_n\right| (\left|\overline{\pi}_L\right|')^{nr}$ over indices $n$
for which $x_n$ is not divisible by $p^{m+1}$.
This is clear for $m=0$. If $m>0$, there is nothing to check unless the supremum is only achieved in cases when $x_n$
is divisible by $p^{m}$. In that case, lift the reduction $\overline{x}$ of $x$ modulo $p$ as in (a) to some $y$
with $\left|y - [\overline{x}]\right|_r < \left|y\right|_r$ for $r \in (0,r_0]$. For $z = (x-y)/p = \sum_n z_n \pi_L^n$, the
supremum in question is also the supremum of $p^{-1} \left|z_n\right| (\left|\overline{\pi}_L\right|')^{nr}$ over indices $n$
for which $z_n$ is not divisible by $p^m$.
By the induction hypothesis, this is at most
\[
p^{-1} \left|z\right|_r = \left|pz\right|_r \leq \max\{\left|x\right|_r, \left|y\right|_r\} = \left|x\right|_r.
\]
This completes the induction, yielding (b).
\end{proof}
This gives us a concrete description of $R^\dagger$ in terms of the coefficients of a series representation.
\begin{cor} \label{C:describe dagger rings}
With notation as in Lemma~\ref{L:describe dagger rings},
$x \in R^\dagger$ if and only if there exists $r>0$ such that
$\sup_n \{\left|x_n\right| (\left|\overline{\pi}_L\right|')^{nr}\} < +\infty$.
\end{cor}
\begin{cor} \label{C:good lifts}
There exists $r_0 > 0$ (depending on $K$) such that every $\overline{x} \in \bA_K/(p)$ admits a lift
$x \in \bA_K^\dagger$ with $\left|x - [\overline{x}]\right|_{r} < \left|x\right|_{r}$ for all $r \in (0, r_0]$.
\end{cor}
\begin{proof}
Apply Lemma~\ref{L:describe dagger rings}(a) to each connected component of $\bA_K$.
\end{proof}

\begin{exercise}
With notation as in Lemma~\ref{L:describe dagger rings}, one can choose $r_0$
so that for $r \in (0,r_0]$, $x \in W^{r-}(L)$ if and only if $\sup_n \{\left|x_n\right| (\left|\overline{\pi}_L\right|')^{nr}\} < +\infty$.
\end{exercise}


\begin{remark} \label{R:connected components}
Beware that our notations do not agree with \cite{brinon-conrad} or most other references when $K \neq \Qp$.
It is more customary to take $\bA_K$ to be the finite \'etale
algebra over $\bA_{\Qp}$ corresponding to a connected component of 
$K \otimes_{\Qp} \Qp(\mu_{p^\infty})$, and similarly for the other
period rings. These rings inherit an action of $\varphi$ and of the subgroup $\Gamma_K$ of $\Gamma$ fixing 
a component of $K_\infty$; the action of $\Gamma$ on $\bA_K$ acts transitively on connected components.
Our point of view has the mild advantage of making the relationship between $K$ and $\bA_K$ more uniform;
for instance, for $L$ a finite Galois extension of $K$, $\Gal(L/K)$ acts on $\bA_L$ with fixed ring
$\bA_K$. We leave to the reader the easy task of translating back and forth between statements in terms of the usual
rings and the corresponding statements in our language.
\end{remark}

\subsection{\'Etale $(\varphi, \Gamma)$-modules}

\begin{defn}
Let $R$ be any of $\tilde{\bA}_K, \bA_K, \tilde{\bA}^{\dagger}_K, \bA^\dagger_K$.
Let $M$ be a finite free $R$-module.
A \emph{semilinear action} of $\varphi$ on $M$ is an additive map $\varphi: M \to M$ for which $\varphi(rm) = \varphi(r) \varphi(m)$ for all $r \in R, m \in M$. Such an action is \emph{\'etale}
if it takes some basis of $M$ to another basis; the same is then true of any basis
(by Remark~\ref{R:action to linear} or Remark~\ref{R:skew conjugation}).
An \emph{\'etale $\varphi$-module} over $R$ is a finite free $R$-module $M$ equipped with
an \'etale semilinear action of $\varphi$.

We similarly define semilinear actions of $\gamma \in \Gamma$. To define a semilinear action of $\Gamma$ as a whole, we insist
that the actions of individual elements compose: for all $\gamma_1, \gamma_2 \in \Gamma$ and $m \in M$, we must have
$\gamma_1(\gamma_2(m)) = (\gamma_1 \gamma_2)(m)$. We say an action of $\Gamma$ is \emph{continuous} if the action map
$\Gamma \times M \to M$ is continuous for the weak topology and (when available) the LF topology.
An \emph{\'etale $(\varphi, \Gamma)$-module} over $R$ is an \'etale $\varphi$-module $M$
equipped with a continuous action of $\Gamma$ commuting with $\varphi$.
(The continuity condition can be omitted; see Exercise~\ref{exer:no continuity}.)
\end{defn}

\begin{remark} \label{R:action to linear}
Define $\varphi^*M = R \otimes_R M$ as a left $R$-module where the left tensorand $R$ is viewed as a left $R$-module in the usual way
and as a right $R$-module via $\varphi$; that is, we have $1 \otimes rm = \varphi(r) \otimes m$ and
$r(s \otimes m) = rs \otimes m$.
One may then view a semilinear action of $\varphi$ as an $R$-linear map $\Phi: \varphi^* M \to M$;
the action is \'etale if and only if $\Phi$ is an isomorphism.
\end{remark}
 
\begin{remark} \label{R:skew conjugation}
A semilinear $\varphi$-action can be specified in terms of a basis $\be_1,\dots,\be_d$ by exhibiting the matrix
$A$ for which $\varphi(\be_j) = \sum_i A_{ij} \be_i$ (which we sometimes call the \emph{matrix of action} of $\varphi$
on the basis). If $\be'_1,\dots,\be'_d$ is another basis, then there exists
an invertible matrix $U$ with $\be'_j = \sum_i U_{ij} \be_i$, and the matrix of action of $\varphi$
on the new basis is $U^{-1} A \varphi(U)$.
\end{remark}

The following lemma has its origins in a construction of Lang; see for example
\cite[Expos\'e XXII, Proposition 1.1]{sga7-2}.

\begin{lemma} \label{L:trivialize}
Let $M$ be an \'etale $(\varphi, \Gamma)$-module over $\tilde{\bA}_K$.
Then for each positive integer $n$, there exists a finite extension $L$ of $K$ for which
\[
(M \otimes_{\tilde{\bA}_K} \tilde{\bA}_L/(p^n))^{\varphi,\Gamma} \otimes_{\Zp} 
\tilde{\bA}_L \to M \otimes_{\tilde{\bA}_K} \tilde{\bA}_L/(p^n)
\]
is an isomorphism.
\end{lemma}
\begin{proof}
Let $\be_1,\dots,\be_d$ be a basis of $M$, and define $A \in \GL_d(\tilde{\bA}_K)$ by
$\varphi(\be_j) = \sum_i A_{ij} \be_i$. For each $\gamma \in \Gamma$, define
$G_\gamma \in \GL_d(\tilde{\bA}_K)$ by 
$\gamma(\be_j) = \sum_i G_{\gamma,ij} \be_i$; then the fact that $\varphi \circ \gamma = \gamma \circ \varphi$
implies that $A \varphi(G_\gamma) = G_\gamma \gamma(A)$. Put
\begin{align*}
R_n &= (\tilde{\bA}_K/(p))[X^{p^{-\infty}}_{ij,k}: i,j=1,\dots,d; k=0,\dots,n-1] \\
S_n &= W(R_n)/(p^n, \varphi^m(A \varphi(X) - X) \,
(m \in \ZZ)),
\end{align*}
where $X$ denotes the matrix with $X_{ij} = \sum_{k=0}^{n-1} p^k [X_{ij,k}]$.
Note that $S_n$ carries an action of $\Gamma$ with $\gamma \in \Gamma$ sending $X$ to $G_\gamma \gamma(X)$.

It can be shown that $S_n$ is finite \'etale over $\tilde{\bA}_K/(p^n)$. In the case $n=1$,
this reduces to observing that 
\[
\tilde{\bA}_K/(p)[X_{ij}:i,j=1,\dots,d]/(A \varphi(X)-X)
\]
is \'etale because the derivative of $\varphi$ is 0. For $n > 1$, an induction argument reduces one to
checking that Artin-Schreier equations define finite \'etale algebras in characteristic $p$.

Thus via Theorem~\ref{T:field of norms specific},
$S_n$ corresponds to a finite \'etale algebra over $K$,
any connected component of which has the desired effect.
\end{proof}

We now arrive at Fontaine's original theorem on $(\varphi, \Gamma)$-modules
\cite{fontaine-phigamma}.
\begin{theorem} \label{T:Fontaine}
The following categories are equivalent.
\begin{enumerate}
\item[(a)]
The category of continuous representations of $G_K$ on finite free $\Zp$-modules.
\item[(b)]
The category of \'etale $(\varphi, \Gamma)$-modules over $\tilde{\bA}_K$.
\item[(c)]
The category of \'etale $(\varphi, \Gamma)$-modules over $\bA_K$.
\end{enumerate}
More precisely, the functor from (c) to (b) is base extension.
\end{theorem}
\begin{proof}
The functor from (a) to (c) is defined as follows.
Let $T$ be a finite free $\Zp$-module, and let $\tau: G_K \to \GL(T)$ be a continuous homomorphism.
For each positive integer $n$, the map $G_K \to \GL(T/p^nT)$ factors through $G_{L/K}$ for some finite Galois extension
$L$ of $K$.
Put
\[
M_n = (T \otimes_{\Zp} \bA_L/(p^n))^{G_{L/K}};
\]
then $\varphi$ and $\Gamma$ act on $\bA_L/(p^n)$ and hence on $M_n$,
and faithfully flat descent for modules (or a more elementary Galois descent) implies that
\begin{equation} \label{eq:isomorphism2}
M_n \otimes_{\bA_K} \bA_L \to T \otimes_{\Zp} \bA_L/(p^n)
\end{equation}
is an isomorphism.
Hence $M = \varprojlim M_n$ is an \'etale $(\varphi, \Gamma)$-module over $\bA_K$.

The functor from (b) to (a) is defined as follows.
Let $M$ be an \'etale $(\varphi, \Gamma)$-module over $\tilde{\bA}_K$.
For each positive integer $n$, choose a finite Galois extension $L$ of $K$
as in Lemma~\ref{L:trivialize} so that
\[
T_n = (M \otimes_{\tilde{\bA}_K} \tilde{\bA}_L/(p^n))^{\varphi,\Gamma}
\]
has the property that the natural map
\begin{equation} \label{eq:isomorphism1}
T_n \otimes_{\Zp} 
\tilde{\bA}_L \to M \otimes_{\tilde{\bA}_K} \tilde{\bA}_L/(p^n)
\end{equation}
is an isomorphism. Note that $G_{L/K}$ acts on $\tilde{\bA}_L$ and hence on $T_n$;
hence $T = \varprojlim T_n$ is a continuous representation of $G_K$.

Using the fact that \eqref{eq:isomorphism2} and \eqref{eq:isomorphism1} are isomorphisms,
it is straightforward
to check that composing around the circle always gives a functor naturally isomorphic to the identity functor
at the starting point.
This completes the proof.
\end{proof}

\subsection{Overconvergence and $(\varphi, \Gamma$)-modules, part 1}

\begin{remark} \label{R:hom}
For $* \in \{\tilde{\bA}, \bA, \tilde{\bA}^\dagger, \bA^\dagger\}$,
for two \'etale $\varphi$-modules (resp.\ \'etale $(\varphi, \Gamma)$-modules) over $*_K$,
we may view $\Hom_{*_K}(M, N)$ naturally as 
an \'etale $\varphi$-module (resp.\ an \'etale $(\varphi, \Gamma)$-module) over $*_K$
by imposing the conditions that
\[
\varphi(f)(\varphi(\be)) = \varphi(f(\be)), \qquad \gamma(f)(\gamma(\be)) = \gamma(f(\be))
\qquad (\gamma \in \Gamma, f \in \Hom_{*_K}(M,N), \be \in M).
\]
The morphisms $M \to N$ of $\varphi$-modules (resp.\ of $(\varphi, \Gamma)$-modules) then are precisely
the elements of $\Hom_{*_K}(M, N)$ fixed by $\varphi$ (resp.\ fixed by $\varphi$ and $\Gamma$).
\end{remark}

\begin{lemma} \label{L:fully faithful1}
Base extension of \'etale $\varphi$-modules which are trivial modulo $p$ from
$\tilde{\bA}_K^\dagger$ to $\tilde{\bA}_K$ is fully faithful.
\end{lemma}
\begin{proof}
By Remark~\ref{R:hom}, this reduces to checking that if $M$ is an \'etale $\varphi$-module over 
$\tilde{\bA}_K^\dagger$ which is trivial modulo $p$, then 
\[
M^{\varphi} = (M \otimes_{\tilde{\bA}_K^\dagger} \tilde{\bA}_K)^{\varphi}.
\]
Let $\be_1,\dots,\be_d$ be a basis of $M$ which is fixed modulo $p$.
Let $L$ be a connected component of $\tilde{\bA}_K/(p)$. Using Theorem~\ref{T:Fontaine},
we can produce an analytic field $L'$ with an isometric embedding $L \hookrightarrow L'$ for which
$M \otimes_{\tilde{\bA}_K^\dagger} W(L')$ admits a $\varphi$-invariant basis
$\be'_1,\dots,\be'_d$.

Let $A \in \GL_d(\tilde{\bA}_K^\dagger)$ be given by $\varphi(\be_j) = \sum_i A_{ij} \be_i$;
then $A-1$ is divisible by $p$. By \eqref{eq:limsup},
we can choose $r>0$ so that $\left|A-1\right|_r < p^{-1/2}$.
Let $U \in \GL_d(W(L))$ be given by $\be'_j = \sum_i U_{ij} \be_i$.
We claim that for each positive integer $n$, $U$ is congruent modulo $p^n$ to some
$V_n \in \GL_d(W^{pr-}(L'))$ with $\left|V_n-1\right|_r, \left|V_n-1\right|_{pr} \leq p^{-1/2}$. This is clear for $n=1$ by taking $V_n=1$.
Given the claim for some $n$, $U$ is congruent modulo $p^{n+1}$ to a matrix $V_n + p^n X$ in which
each entry $X_{ij}$ is a Teichm\"uller lift. We have
\[
\varphi(X) - X \equiv p^{-n}(V_n - \varphi(V_n) - (A-1)\varphi(V_n)) \pmod{p},
\]
which implies that $\left|\varphi(X_{ij}) - X_{ij}\right|_r \leq p^{n-1/2}$. From this it follows in turn that
\[
\left|X_{ij}\right|_{pr} \leq p^{n-1/2}.
\]
We may then take $V_{n+1} = V_n + p^n X$ for the desired effect.

{}From the previous paragraph, it follows that 
$\be'_1,\dots,\be'_d$ form a basis of $M \otimes_{\tilde{\bA}_K^\dagger} W^\dagger(L')$.
By expressing a $\varphi$-invariant element of $M$ using this basis
and noting that
\[
W(L')^{\varphi} = W((L')^{\varphi}) = \ZZ_p \subseteq 
W^\dagger(L')^{\varphi},
\]
we see that
the image of $(M \otimes_{\tilde{\bA}_K^\dagger} \tilde{\bA}_K)^{\varphi}$
in $M \otimes_{\tilde{\bA}_K^\dagger} W(L')$ is contained in
$M \otimes_{\tilde{\bA}_K^\dagger} W^\dagger(L')$.
Because
$W(L) \cap W^\dagger(L') = W^\dagger(L)$ and $M$ is a free module,
we may further conclude that the image lands in
$M \otimes_{\tilde{\bA}_K^\dagger} W^\dagger(L)$. 
Since this is true for each $L$, we deduce that 
$(M \otimes_{\tilde{\bA}_K^\dagger} \tilde{\bA}_K)^{\varphi} \subseteq M$, as desired.
\end{proof}

\begin{lemma} \label{L:keep continuity}
Let $M^\dagger$ be an \'etale $\varphi$-module over $\tilde{\bA}_K^\dagger$ which is trivial modulo $p$,
and suppose that
$M^\dagger \otimes_{\tilde{\bA}_K^\dagger} \tilde{\bA}_K$ carries a semilinear action of $\Gamma$ 
which commutes with $\varphi$
and is continuous for the weak topology.
Then the action of $\Gamma$ on $M^\dagger$ provided by Lemma~\ref{L:fully faithful1} is continuous for the LF topology.
\end{lemma}
\begin{proof}
Retain notation as in the proof of Lemma~\ref{L:fully faithful1}. Note that $\Gamma$ acts on
$M^\dagger \otimes_{\tilde{\bA}_K^\dagger} W(L')$ continuously for the weak topology.
It thus acts on
$(M^\dagger \otimes_{\tilde{\bA}_K^\dagger} W(L'))^{\varphi} =
(M^\dagger \otimes_{\tilde{\bA}_K^\dagger} W^\dagger(L'))^{\varphi}$ continuously for the weak topology.
However, the latter is a finite free $\Zp$-module on which the weak and LF topologies coincide.
We thus obtain a continuous action of $\Gamma$ on $M^\dagger \otimes_{\tilde{\bA}_K^\dagger} W^\dagger(L')$
for the LF topology. Since this is true for each $L$, the action of $\Gamma$ on
$M^\dagger$ is continuous for the LF topology, as desired.
\end{proof}

\begin{lemma} \label{L:good basis}
Let $L$ be an analytic field which is perfect of characteristic $p$, with norm $\left|\bullet \right|'$.
Let $M$ be a finite free $W(L)$-module equipped with an \'etale semilinear $\varphi$-action
and admitting a basis which is fixed by $\varphi$ modulo $p$.
Then there exists a basis of $W(L)$ on which $\varphi$ acts via an invertible matrix over $W^\dagger(L)$.
\end{lemma}
\begin{proof}
Let $\be_1,\dots,\be_d$ be a basis of $M$ which is fixed modulo $p$,
and let $F \in \GL_d(W(L))$ be defined by $\varphi(\be_j) = \sum_i F_{ij} \be_i$.
We construct sequences of matrices $F_n, G_n$ such that $F_1 = F$, $G_1 = 1$,
$F_n-1$ has entries in $p W(L)$, $G_n$ has entries in $W^{1-}(L)$,
$\left|G_n-1\right|_1 < 1$, and $X_n = p^{-n}(F_n - G_n)$ has entries in $W(L)$.
Given $F_n$ and $G_n$, choose a nonnegative integer $m$ for which
$\left|\varphi^{-m}(\overline{X}_n)\right|' < p^{n/2}$,
put $\overline{Y}_n = -\sum_{h=1}^{m} \varphi^{-h}(\overline{X}_n)$,
and put
\[
U_{n} = 1 + p^n [\overline{Y}_n], \quad
F_{n+1} = U_{n}^{-1} F_n \varphi^d(U_{n}), \quad
G_{n+1} = G_n + p^n [\varphi^{-m}(\overline{X}_n)],
\]
where the Teichm\"uller map is applied to matrices entry by entry.
The product $U_1 U_2 \cdots$ converges $p$-adically to a matrix $U$ for which
$U^{-1} F \varphi^d(U)$ is equal to the $p$-adic limit of the $G_n$,
which is invertible over $W^{1-}(L)$.
Thus the vectors $\be'_j = \sum_i U_{ij} \be_i$ form a basis of the desired form.
\end{proof}

\begin{theorem} \label{T:overconvergent1}
Base extension of \'etale $(\varphi, \Gamma)$-modules from
$\tilde{\bA}_K^\dagger$ to $\tilde{\bA}_K$ is an equivalence of categories. Consequently (by Theorem~\ref{T:Fontaine}),
both categories are equivalent to the category of continuous representations of $G_K$ on finite free $\Zp$-modules.
\end{theorem}
\begin{proof}
We first check that the base extension functor is fully faithful. Again 
by Remark~\ref{R:hom}, this reduces to checking that if $M$ is an \'etale $(\varphi, \Gamma)$-module over 
$\tilde{\bA}_K^\dagger$, then 
\[
M^{\varphi, \Gamma} = (M \otimes_{\tilde{\bA}_K^\dagger} \tilde{\bA}_K)^{\varphi, \Gamma}.
\]
By Theorem~\ref{T:Fontaine}, $M \otimes_{\tilde{\bA}_K^\dagger} \tilde{\bA}_K$ corresponds to a 
continuous representation of $G_K$ on a finite free $\Zp$-module $T$. We can then find a finite extension
$L$ of $K$ such that $G_L$ acts trivially on $T/pT$; this means that $M/pM \otimes_{\tilde{\bA}_K^\dagger/(p)}
\tilde{\bA}_L^\dagger/(p)$ admits a basis fixed by both $\varphi$ and $\Gamma$.
By Lemma~\ref{L:fully faithful1},
\[
(M \otimes_{\tilde{\bA}_K^\dagger} \tilde{\bA}_K)^{\varphi, \Gamma} \subseteq (M \otimes_{\tilde{\bA}_K^\dagger} \tilde{\bA}_L)^{\varphi, \Gamma}
= (M \otimes_{\tilde{\bA}_K^\dagger} \tilde{\bA}^\dagger_L)^{\varphi, \Gamma}.
\]
However, within $\tilde{\bA}_L$ we have
\[
\tilde{\bA}_L^\dagger \cap \tilde{\bA}_K = \tilde{\bA}^\dagger_K
\]
and likewise after tensoring with the finite free module $M$. Hence
$(M \otimes_{\tilde{\bA}_K^\dagger} \tilde{\bA}_K)^{\varphi, \Gamma}
\subseteq M^{\varphi, \Gamma}$
as desired.

It remains to check that  the base extension functor is essentially surjective.
For a given \'etale $(\varphi, \Gamma)$-module $M$ over $\tilde{\bA}_K$, to show that $M$ descends to
$\tilde{\bA}^{\dagger}_K$, it is enough to check that it descends to
$\tilde{\bA}^{\dagger}_L$ for some finite Galois extension $L$ of $K$
(as then Lemma~\ref{L:fully faithful1} provides the data needed to perform Galois descent back to $K$).
Consequently, using Theorem~\ref{T:Fontaine}, we may again reduce to the case where
$M$ admits a basis $\be_1,\dots,\be_d$ which is fixed by $\varphi$ and $\Gamma$ modulo $p$.

By Lemma~\ref{L:good basis} applied to each component of $\tilde{\bA}_K/(p)$,
we obtain an \'etale $\varphi$-module $M^\dagger$ over $\tilde{\bA}^\dagger_K$ and an isomorphism
$M^\dagger \otimes_{\tilde{\bA}^\dagger_K} \tilde{\bA}_K \cong M$ of $\varphi$-modules. By
Lemma~\ref{L:fully faithful1} and Lemma~\ref{L:keep continuity}, 
the action of $\Gamma$ on $M$ induces an action on $M^\dagger$ which
is continuous for the weak and LF topologies.
This yields the desired result.
\end{proof}

\begin{exercise} \label{exer:no continuity}
One can omit the requirement of continuity of the action of $\Gamma$ in the definition of an
\'etale $(\varphi, \Gamma)$-module, as it is implied by the other properties.
(Hint: first do the case over $\tilde{\bA}_K$, which implies the case over $\bA_K$.
Then use Theorem~\ref{T:overconvergent1} to deduce the case over $\tilde{\bA}^\dagger_K$,
which implies the case over $\bA^\dagger_K$.)
\end{exercise}

\subsection{More on the action of $\Gamma$}

We have seen that descent of \'etale $(\varphi, \Gamma)$-modules from $\tilde{\bA}_K$ to 
$\tilde{\bA}^\dagger_K$ can be achieved mainly using the bijectivity of $\varphi$. 
To make a similar passage from $\bA_K$ to $\bA^\dagger_K$, we trade the failure of this bijectivity
for better control of the action of $\Gamma$.

\begin{lemma} \label{L:gamma analytic}
There exists $c>0$ such that for $\overline{x} \in \bA_K/(p)$, for $n$ a positive integer, for
$\gamma \in 1 + p^n \Zp \subseteq \Gamma$,
\[
\left|(\gamma-1)(\overline{x})\right|' \leq cp^{-p^{n+1}/(p-1)} \left|\overline{x}\right|'.
\]
\end{lemma}
\begin{proof}
For $K = \Qp$, for $\overline{x} = \overline{\pi}$,
\[
(\gamma-1)(\overline{\pi}) = (\gamma-1)(1 + \overline{\pi}) = (1 + \overline{\pi})((1 + \overline{\pi})^{\gamma-1} - 1)
\]
is divisible by $\overline{\pi}^{p^n}$, so $\left|(\gamma-1)(\overline{\pi})\right|' \leq (\left|\overline{\pi}\right|')^{p^n} 
= p^{-p^{n+1}/(p-1)} \left|\overline{\pi}\right|'$. This implies the general result for $K = \Qp$ with $c=1$.

For general $K$, it is similarly sufficient to check the claim for $\overline{x} = \overline{\pi}_L$ a uniformizer of a connected component
$L$ of $\bA_K/(p)$. Let $P(T) = \sum_i P_i T^i$ be the minimal polynomial of $\overline{\pi}_L$ over $\Fp((\overline{\pi}))$,
and put $Q(T) = P(T+\overline{\pi}_L) = \sum_{i>0} Q_i T^i$ with $Q_1 \neq 0$. Then 
\[
0 = \gamma(P(\overline{\pi}_L))= (\gamma-1)(P)(\gamma(\overline{\pi}_L)) + P(\gamma(\overline{\pi}_L)) \\
= \sum_i (\gamma-1)(P_i) \gamma(\overline{\pi}_L)^i + Q(\gamma(\overline{\pi}_L)-\overline{\pi}_L).
\]
Since $\gamma$ acts continuously on $\bA_K/(p)$, $\left|(\gamma-1)(\overline{\pi}_L)\right|' \to 0$ as $n \to \infty$;
hence for large $n$, $\left|Q(\gamma(\overline{\pi}_L)-\overline{\pi}_L)\right|' = \left|Q_1\right|'\left|(\gamma-1)(\overline{\pi}_L)\right|'$. It follows that for large $n$,
\[
\left|(\gamma-1)(\overline{\pi}_L)\right|' = \left|Q_1^{-1}\right|' \left| \sum_i (\gamma-1)(P_i) \gamma(x)^i \right|' \leq p^{-p^{n+1}/(p-1)}
\left|Q_1^{-1}\right|' \max_i \{\left|P_i\right|' \left|\overline{\pi}_L^i\right|'\};
\]
this implies the desired result.
\end{proof}
\begin{remark} \label{R:gamma not analytic}
Note that Lemma~\ref{L:gamma analytic} is a far cry from what happens over $\tilde{\bA}_K/(p)$:
one cannot establish an inequality
of the form $\left|(\gamma-1)(x)\right|' \leq c\left|x\right|'$ uniformly over $x \in \tilde{\bA}_K/(p)$
for even a single choice of $\gamma \in \Gamma - \{1\}$ and $c<1$.
Namely, if one had such an inequality, then substituting $x^{1/p}$ in place of $x$ would immediately imply the same inequality with
$c$ replaced by $c^p$, which would ultimately force $\left|(\gamma-1)(x)\right|' = 0$ for all $x$. This issue is closely related to the question of identifying
\emph{locally analytic vectors} within $(\varphi, \Gamma)$-modules and their generalizations
\cite{berger-analytic}.
\end{remark}

\begin{lemma} \label{L:direct sum}
There exists $c > 0$ (depending on $K$) such that for any positive integer $m$
and any $\overline{x} = \sum_{e=0}^{p^m-1} (1 + \pi)^{e/p^m} \overline{x}_{e/p^m}$
with $\overline{x}_{e/p^m} \in \bA_K/(p)$, 
\[
\max_e\left\{\left|\overline{x}_{e/p^m}\right|'\right\} \leq c \left|\overline{x}\right|'.
\]
\end{lemma}
\begin{proof}
We first produce $c_0 \geq 1$ that satisfies the claim for $m=1$.
To do this, note that $\left|\overline{x}\right|'$ and $\max_e \{\left|\overline{x}_{e/p}\right|'\}$ both define norms on 
the finite-dimensional vector space $\varphi^{-1}(\bA_K/(p))$ over $\Fp((\overline{\pi}))$.
The claim then follows from the fact that any two norms on a finite-dimensional vector space over a complete
field are equivalent (e.g., see \cite[Theorem~1.3.6]{kedlaya-course}).

We next show by induction on $m$ that for any nonnegative integer $m$
and any $\overline{x} = \sum_{e=0}^{p^m-1} (1 + \pi)^{e/p^m} \overline{x}_{e/p^m}$
with $\overline{x}_{e/p^m} \in \bA_K/(p)$, 
\[
\max_e\left\{\left|\overline{x}_{e/p^m}\right|'\right\} \leq c_0^{1 + 1/p + \cdots + 1/p^{m-1}}  \left|\overline{x}\right|'.
\]
This is vacuously true for $m=0$. For $m>0$, if the claim is known for $m-1$,
then by the previous paragraph we can write $\overline{x}^{p^{m-1}} = \sum_{f=0}^{p-1} (1 + \pi)^{f/p} \overline{y}_{f/p}$
with $\overline{y}_{f/p} \in \bA_K/(p)$ and 
$\max_f \{\left|\overline{y}_{f/p}\right|'\} \leq c_0 \left|\overline{x}^{p^{m-1}}\right|'$.
By the induction hypothesis, we can then write
$\overline{y}_{f/p}^{p^{-m+1}} = \sum_{g=0}^{p^{m-1}-1} (1 + \pi)^{g/p^{m-1}} \overline{z}_{f/p, g/p^{m-1}}$ with $\overline{z}_{f/p, g/p^{m-1}} \in \bA_K/(p)$ and
\[
\max_g \left\{\left|\overline{z}_{f/p,g/p^{m-1}}\right|'\right\}  \leq c_0^{1+1/p + \cdots + 1/p^{m-2}} \left|\overline{y}_{f/p}^{p^{-m+1}}\right|'.
\]
For $e = 0,\dots,p^m-1$, write $e = p^{m-1} f + g$ with $f \in \{0,\dots,p-1\}$ and $g \in \{0,\dots,p^{m-1}-1\}$ and set
$\overline{x}_{e/p^m} = \overline{z}_{f/p,g/p^{m-1}}$; then
$\overline{x} = \sum_{e=0}^{p^m-1} (1 + \pi)^{e/p^m} \overline{x}_{e/p^m}$ and
\[
\max_e \left\{\left|\overline{x}_{e/p^m}\right|'\right\} \leq c_0^{1 + 1/p + \cdots + 1/p^{m-2}}  \max_f \left\{\left|\overline{y}_{f/p}^{p^{-m+1}}\right|'\right\}
\leq c_0^{1 + 1/p + \cdots + 1/p^{m-1}} \left|\overline{x}\right|'.
\]
This completes the induction; we may now take
\[
c = c_0^{1 + 1/p + 1/p^2 + \cdots} = c_0^{p/(p-1)}
\]
and deduce the desired result.
\end{proof}
\begin{cor} \label{C:direct sum}
Let $\overline{T} \subset \tilde{\bA}_K/(p)$ be the closure of the subgroup generated by
$(1 + \overline{\pi})^e \bA_K/(p)$ for all $e \in \ZZ[p^{-1}] \cap (0,1)$. Then the natural map
$\bA_K/(p) \oplus \overline{T} \to \tilde{\bA}_K/(p)$ is an isomorphism of Banach spaces over $\Fp((\overline{\pi}))$.
\end{cor}


\begin{lemma}
For $\overline{T}$ as in Corollary~\ref{C:direct sum}, for any $\gamma \in \Gamma - \{1\}$, the map $\gamma-1: \overline{T} \to \overline{T}$ is bijective
with bounded inverse.
\end{lemma}
\begin{proof}
By Lemma~\ref{L:direct sum}, it is sufficient to check that for each $e \in \ZZ[p^{-1}] \cap (0,1)$,
$\gamma-1$ is bijective on $(1 + \overline{\pi})^e \bA_K/(p)$ with the inverse bounded uniformly in $e$. Using the identity
\[
(\gamma-1)^{-1} = (1 + \gamma + \cdots + \gamma^{m-1})(\gamma^m - 1)^{-1},
\]
we may reduce the claim for $\gamma$ to the claim for $\gamma^m$ for any convenient positive integer $m$ (chosen uniformly
in $e$). Consequently, we may assume that $\gamma \in (1 + p^n \Zp) - (1 + p^{n+1} \Zp)$ for $n$ large enough that
there exists a value $c$ as in Lemma~\ref{L:gamma analytic} which is less than $p^{p^n}$.

For $\overline{x} \in \bA_K/(p)$ we may write
\begin{align*}
(\gamma-1)((1 + \overline{\pi})^e \overline{x}) &= (\gamma-1)((1 + \overline{\pi})^e)\overline{x} +
\gamma((1 + \overline{\pi})^{e}) (\gamma-1)(\overline{x}) \\
&= (1 + \overline{\pi})^e((1 + \overline{\pi})^{(\gamma-1)e} - 1) \overline{x}
+ (1 + \overline{\pi})^{\gamma e} (\gamma-1)(\overline{x}).
\end{align*}
By Lemma~\ref{L:gamma analytic}, $\left|(\gamma-1)(\overline{x})\right|'\leq  c p^{-p^{n+1}/(p-1)} \left|\overline{x}\right|'$. On the other hand, since $e \in \ZZ[p^{-1}] \cap (0,1)$, $(\gamma-1)e$ has $p$-adic valuation
at most $n-1$, so $\left|(1 + \overline{\pi})^{(\gamma-1)e} - 1\right|' \geq (\left|\overline{\pi}\right|')^{p^{n-1}}
= p^{-p^n/(p-1)}$. Since $c p^{-p^{n+1}/(p-1)} < p^{-p^n/(p-1)}$ by our choice of $n$,
the operator 
\begin{equation} \label{eq:gamma assign}
\overline{x} \mapsto (1 + \overline{\pi})^{-e} ((1 + \overline{\pi})^{(\gamma-1)e} - 1)^{-1} (\gamma-1)((1 + \overline{\pi})^e \overline{x})
\end{equation}
on $\bA_K/(p)$ is equal to the identity map plus the operator
\[
\overline{x} \mapsto (1 + \overline{\pi})^{(\gamma-1)e} ((1 + \overline{\pi})^{(\gamma-1)e} - 1)^{-1}
(\gamma-1)(\overline{x})
\]
whose norm is less than 1. Therefore, \eqref{eq:gamma assign} is an invertible
operator whose inverse has norm 1.
This proves the claim.
\end{proof}
\begin{cor}  \label{C:invert}
For $\overline{T}$ as in Corollary~\ref{C:direct sum} and $\gamma \in \Gamma-\{1\}$,
every $\overline{x} \in \tilde{\bA}_K/(p)$ can be written uniquely
as $\overline{y} + (\gamma-1)(\overline{z})$ with $\overline{y} \in \bA_K/(p)$,
$\overline{z}\in \overline{T}$. Moreover, 
\[
\max\{\left|\overline{y}\right|', \left|\overline{z}\right|'\} \leq c \left|\overline{x}\right|'
\]
for some constant $c$ (depending on $K$ and $\gamma$ but not on $\overline{x}$).
\end{cor}
\begin{cor} \label{C:lift direct sum invert}
Let $T \subset \tilde{\bA}_K$ be the closure for the weak topology of the subgroup generated by
$(1 + \pi)^e \bA_K$ for all  $e \in \ZZ[p^{-1}] \cap (0,1)$. 
For $\gamma \in \Gamma - \{1\}$, every $x \in \tilde{\bA}_K$ can be written uniquely as 
$y + (\gamma-1)(z)$ with $y \in \bA_K$, $z \in T$. Moreover,
there exist $c,r_0 >0$ (depending on $K$ and $\gamma$) such that
\[
\max\{\left|y\right|_r, \left|z\right|_r\} \leq c^r \left|x\right|_r \qquad (r \in (0,r_0]).
\]
\end{cor}
\begin{proof}
Combine Corollary~\ref{C:invert}
with Corollary~\ref{C:good lifts}.
\end{proof}

\subsection{Overconvergence and $(\varphi, \Gamma)$-modules, part 2: the theorem of Cherbonnier-Colmez}

We start with the following analogue of Lemma~\ref{L:fully faithful1}.
\begin{lemma} \label{L:fully faithful2}
Base extension of \'etale $\varphi$-modules which are trivial modulo $p$ from
$\bA_K^\dagger$ to $\bA_K$ is fully faithful.
\end{lemma}
\begin{proof}
Again by Remark~\ref{R:hom}, this reduces to checking that if $M$ is an \'etale $\varphi$-module over 
$\bA_K^\dagger$  which is trivial modulo $p$, then 
\[
M^{\varphi} = (M \otimes_{\bA_K^\dagger} \bA_K)^{\varphi}.
\]
By Lemma~\ref{L:fully faithful1}, we already have
\[
(M \otimes_{\bA_K^\dagger} \bA_K)^{\varphi} \subseteq
(M \otimes_{\bA_K^\dagger} \tilde{\bA}_K)^{\varphi} =
(M \otimes_{\bA_K^\dagger} \tilde{\bA}^{\dagger}_K)^{\varphi}.
\]
Since $M$ is a free module, within $M \otimes_{\bA_K^\dagger} \tilde{\bA}_K$ we have
\[
(M \otimes_{\bA_K^\dagger} \bA_K)
\cap
(M \otimes_{\bA_K^\dagger} \tilde{\bA}^\dagger_K)
= M,
\]
yielding the desired result.
\end{proof}

\begin{theorem}[Cherbonnier-Colmez] \label{T:overconvergent2}
Base extension of \'etale $(\varphi, \Gamma)$-modules from
$\bA_K^\dagger$ to $\bA_K$ is an equivalence of categories. Consequently (by Theorem~\ref{T:Fontaine}),
both categories are equivalent to the category of continuous representations of $G_K$ on finite free $\Zp$-modules.
\end{theorem}
\begin{proof}
By Theorem~\ref{T:Fontaine} and Theorem~\ref{T:overconvergent1}, it is equivalent to show that
base extension of \'etale $(\varphi, \Gamma)$-modules from
$\bA_K^\dagger$ to $\tilde{\bA}^\dagger_K$ is an equivalence of categories. 
As in the proof of Theorem~\ref{T:overconvergent1} (but now using Lemma~\ref{L:fully faithful2} instead of
Lemma~\ref{L:fully faithful1}), it is sufficient to check that any \'etale $(\varphi,\Gamma)$-module
$M$ over $\tilde{\bA}^\dagger_K$ admitting a basis $\be_1,\dots,\be_d$ fixed by $\varphi$ and $\Gamma$ modulo $p$
descends to $\bA_K^\dagger$. 

Put $\gamma = 1 + p^2 \in \Gamma$.
Define $T,c,r_0$ as in Corollary~\ref{C:lift direct sum invert}.
Let $G \in \GL_d(\tilde{\bA}^\dagger_K)$ be given by $\gamma(\be_j) = \sum_i G_{ij} \be_i$, so that $G-1$ is divisible by $p$.
By \eqref{eq:limsup}, we can choose $r \in (0,r_0]$ so that $\epsilon = \left|G-1\right|_r^{1/3} < \min\{c^{-r},1\}$.

We define a sequence of invertible matrices $U_0, U_1,\dots$ over $\tilde{\bA}^\dagger_K$ congruent to 1 modulo $p$
with the property that $G_l = U_l^{-1} G \gamma(U_l)$ can be written as $1 + X_l + (\gamma-1)(Y_l)$
with $X_l$ having entries in $\bA^\dagger_K$, $Y_l$ having entries in $T$, and
\[
\left|X_l\right|_r \leq \epsilon^2, \,\left|Y_l\right|_r \leq \epsilon^{l+2}.
\]
To begin with, put $U_0 = 1$ and apply Corollary~\ref{C:lift direct sum invert} to construct
$X_0, Y_0$ of the desired form with $\left|X_0\right|_r, \left|Y_0\right|_r \leq c^r \left|G-1\right|_r \leq \epsilon^2$.
Given $U_l$, set $U_{l+1} = U_l(1 - Y_l)$ and write
\begin{align*}
G_{l+1} &= (1 - Y_l)^{-1}(1 + X_l + (\gamma-1)(Y_l))(1 - \gamma(Y_l)) \\
&= 1 + X_l + Y_lX_l - X_l \gamma(Y_l) + E_l
\end{align*}
with $\left|E_l\right|_r \leq \epsilon^{2l+4}$. Note that $Y_l X_l - X_l \gamma(Y_l)$ has entries in $T$
and $\left|Y_l X_l - X_l \gamma(Y_l)\right|_r \leq \epsilon^{l+4}$. 
Apply Corollary~\ref{C:lift direct sum invert} to split $Y_lX_l - X_l \gamma(Y_l) + E_l$ as $A_l + (\gamma-1)(B_l)$
with $A_l$ having entries in $\bA^\dagger_K$, $B_l$ having entries in $T$,
and $\left|A_l\right|_r, \left|B_l\right|_r \leq c^r \epsilon^{l+4} \leq \epsilon^{l+3}$. 
Set $X_{l+1} = X_l + A_l$, $Y_{l+1} =  B_l$ and continue.

The product $U= U_0 U_1 \cdots$ converges to an invertible matrix $U$ over $\tilde{\bA}^{\dagger}_K$.
Define the basis $\be'_1,\dots,\be'_d$ of $M$ by $\be'_j = \sum_i U_{ij} \be_i$.
Define the matrices $A,H$ by $\varphi(\be'_j) = \sum_i A_{ij} \be_i$,
$\gamma(\be'_j) = \sum_i H_{ij} \be_i$. By construction, $H$ has entries in $\bA^\dagger_K$
and is congruent to 1 modulo $p$. Since $\varphi$ and $\gamma$ commute,
$A \varphi(H)= H \gamma(A)$.
Apply Corollary~\ref{C:lift direct sum invert} to write $A = B + C$ with $B$ having entries in
$\bA^\dagger_K$ and $C$ having entries in $T$; then
\[
H^{-1} C \varphi(H) - C = (\gamma-1)(C).
\]
If $C$ is nonzero, then there is a largest nonnegative integer $m$ such that $C$ is divisible by $p^m$.
However, since $H \equiv 1 \pmod{p}$, $H^{-1} C \varphi(H) - C$ is divisible by $p^{m+1}$
while $(\gamma-1)(C)$ is not, a contradiction. Hence $C = 0$ and $A = B$ has entries in
$\bA^\dagger_K$.

Let $M^\dagger$ be the $\bA^\dagger_K$-span of $\be'_1,\dots,\be'_d$; it is an \'etale
$\varphi$-module over $\bA^\dagger_K$ such that $M^\dagger \otimes_{\bA^\dagger_K} \tilde{\bA}^\dagger_K
\cong M$. By Lemma~\ref{L:fully faithful2}, the action of $\Gamma$
descends to $M^\dagger$; it is automatically continuous because $\bA^\dagger_K$ and $\tilde{\bA}^\dagger_K$
carry the same topologies. This proves the desired result.
\end{proof}

\begin{remark}
In \cite{cherbonnier-colmez} and elsewhere, Theorem~\ref{T:overconvergent2} is described as the statement that $p$-adic Galois representations are \emph{overconvergent}. This term refers to the distinction between $\bA_{\Qp}$ and $\bA^\dagger_{\Qp}$. Recall that $\bA_{\Qp}$ consists of formal Laurent series in $\pi$ with coefficients in $\Zp$ such that the coefficient of $\pi^n$ converges $p$-adically to 0 as $n \to -\infty$ (with no restriction as $n \to +\infty$. That is, the negative part of the series converges on the disc $\left| \pi^{-1} \right| \leq 1$.
By contrast, by Corollary~\ref{C:describe dagger rings} such a series belongs to $\bA^{\dagger}_{\Qp}$ if and only if the $p$-adic valuation of the coefficient of $\pi^n$ grows at least linearly in $-n$ as $n \to -\infty$; that is, the negative part of the series converges on a disc of the form $\left| \pi^{-1} \right| \leq 1+ \epsilon$ for some $\epsilon > 0$.
\end{remark}


\begin{thebibliography}{99}

\bibitem{andreatta-brinon}
F. Andreatta and O. Brinon,
Surconvergence des repr\'esentations $p$-adiques: le cas relatif,
\textit{Ast\'erisque} \textbf{319} (2008), 39--116.


\bibitem{berger-analytic}
L. Berger, Multivariable $(\varphi, \Gamma)$-modules
and locally analytic vectors,
preprint (2015) available at \url{http://perso.ens-lyon.fr/laurent.berger/}.

\bibitem{berger-colmez}
L. Berger and P. Colmez,
Familles de repr\'esentations de de Rham et monodromie $p$-adique,
\textit{Ast\'erisque} \textbf{319} (2008), 303--337.

\bibitem{brinon-conrad}
O. Brinon and B. Conrad, Lecture notes on $p$-adic Hodge theory,
available at \url{http://math.stanford.edu/~conrad/}.

\bibitem{cherbonnier-colmez}
F. Cherbonnier and P. Colmez,
R\'epresentations $p$-adiques surconvergentes,
\textit{Invent. Math.} \textbf{133} (1998), 581--611.

\bibitem{coleman-iovita}
R. Coleman and A. Iovi\c{t}a, Frobenius and monodromy operators for
curves and abelian varieties, \textit{Duke Math. J.} \textbf{97} (1999), 171--217.

\bibitem{colmez-langlands}
P. Colmez, Repr\'esentations de $\GL_2(\Qp)$ et $(\varphi, \Gamma)$-modules,
\textit{Ast\'erisque} \textbf{330} (2010), 281--509.

\bibitem{sga7-2}
P. Deligne and N. Katz, \textit{Groupes de Monodromie en
G\'eometrie Alg\'ebriques (SGA 7)}, second part, Lecture Notes
in Math.\ 340, Springer-Verlag, Berlin, 1973.

\bibitem{fargues-fontaine}
L. Fargues and J.-M. Fontaine, Courbes et fibr\'es vectoriels en th\'eorie de Hodge
$p$-adique, in preparation; draft (2013)
available at \url{http://webusers.imj-prg.fr/~laurent.fargues/}.

\bibitem{fontaine-phigamma}
J.-M. Fontaine,
Repr\'esentations $p$-adiques des corps locaux, I,
in \textit{The Grothendieck Festschrift, Vol. II},
Progr. Math. 87, Birkh\"auser, Boston,
1990, 249--309.

\bibitem{fontaine-wintenberger}
J.-M. Fontaine and J.-P. Wintenberger,
Le ``corps des normes'' de certaines extensions alg\'ebriques
de corps locaux,
\textit{C.R. Acad. Sci. Paris S\'er. A-B} \textbf{288} (1979), A367--A370.

\bibitem{gabber-ramero}
O. Gabber and L. Ramero,
\textit{Almost Ring Theory}, Lecture Notes in Math. 1800,
Springer-Verlag, Berlin, 2003.

\bibitem{kedlaya-power}
K.S. Kedlaya, Power series and $p$-adic algebraic closures,
\textit{J. Num. Theory} \textbf{89} (2001), 324--339.

\bibitem{kedlaya-course}
K.S. Kedlaya, \textit{$p$-adic Differential Equations},
Cambridge Studies in Advanced Math.\ 125, Cambridge Univ. Press, Cambridge,
2010.

\bibitem{kedlaya-icm}
K.S. Kedlaya,
Relative $p$-adic Hodge theory and Rapoport-Zink period domains,
in \textit{Proceedings of the International Congress of Mathematicians
(Hyderabad, 2010), Volume II}, Hindustan Book Agency, 2010, 258--279.

\bibitem{kedlaya-witt}
K.S. Kedlaya, Nonarchimedean geometry of Witt vectors,
\textit{Nagoya Math. J.} \textbf{209} (2013), 111--165.

\bibitem{kedlaya-noetherian}
K.S. Kedlaya, Noetherian properties of Fargues-Fontaine curves,
arXiv:1410.5160v1 (2014).

\bibitem{kedlaya-liu1}
K.S. Kedlaya and R. Liu, Relative $p$-adic Hodge theory, I: Foundations, arXiv:1301.0792v4 (2014); to appear in \textit{Ast\'erisque}.

\bibitem{kedlaya-liu2}
K.S. Kedlaya and R. Liu, Relative $p$-adic Hodge theory, II: Imperfect period rings, in preparation.

\bibitem{raynaud}
M. Raynaud, \textit{Anneaux Locals Hens\'eliens}, Lecture Notes in Math. 169, 
Springer-Verlag, Berlin, 1970.

\bibitem{scholze1}
P. Scholze, Perfectoid spaces,
\textit{Publ. Math. IH\'ES} \textbf{116} (2012), 245--313.

\bibitem{scholze2}
P. Scholze, $p$-adic Hodge theory for rigid analytic varieties,
\textit{Forum of Math. Pi} \textbf{1} (2013), e1.

\bibitem{scholze-icm}
P. Scholze, Perfectoid spaces and their applications,
in \textit{Proceedings of the International Congress of Mathematicians, August 13--21, 2014, Seoul}.

\bibitem{scholze-torsion}
P. Scholze, On torsion in the cohomology of locally symmetric varieties, 
preprint (2013) available at
\url{http://www.math.uni-bonn.de/people/scholze/}.

\bibitem{scholze-weinstein}
P. Scholze and J. Weinstein, Moduli of $p$-divisible groups,
\textit{Cambridge J. Math.} \textbf{1} (2013), 145--237.

\bibitem{sen-lie}
S. Sen, Ramification in $p$-adic Lie extensions,
\textit{Inv. Math.} \textbf{17} (1972), 44--50.

\bibitem{serre-local-fields}
J.-P. Serre, \textit{Local Fields}, Graduate Texts in Math. 67, 
Springer-Verlag, New York, 1979.

\end{thebibliography}
\end{document}